\documentclass[12pt,a4paper]{article}
\usepackage{amsmath,amssymb,amsthm,bm}
\usepackage{graphicx}
\usepackage{txfonts}

\newtheorem{theorem}{Theorem}[section]
\newtheorem{proposition}[theorem]{Proposition}

\newtheorem{lemma}[theorem]{Lemma}
\newtheorem{remark}[theorem]{Remark}
\newtheorem{example}[theorem]{Example}

\newcommand{\1}{{\bm 1}}

\numberwithin{equation}{section}

\begin{document}

\begin{center}
\large\bf
Quadratic Embedding Constants of \\
Fan Graphs and Graph Joins
\end{center}

\bigskip

\begin{center}
Wojciech M{\l}otkowski \\
Instytut Matematyczny \\
Uniwersytet Wroc{\l}awski \\
Plac Grunwaldzki 2/4, 50-384 Wroc{\l}aw, Poland \\
mlotkow@math.uni.wroc.pl
\end{center}

\begin{center}
Nobuaki Obata\\
Center for Data-driven Science and Artificial Intelligence \\
Tohoku University\\
Sendai 980-8576 Japan \\
and \\
Combinatorial Mathematics Research Group \\
Faculty of Mathematics and Natural Sciences \\
Institut Teknologi Bandung
Jalan Ganesa 10 Bandung, Indonesia \\
obata@tohoku.ac.jp
\end{center}

\bigskip

\begin{quote}
\textbf{Abstract}\enspace
We derive a general formula for the quadratic embedding constant
of a graph join $\bar{K}_m+G$,
where $\bar{K}_m$ is the empty graph on $m\ge1$ vertices
and $G$ is an arbitrary graph.
Applying our formula to a fan graph $K_1+P_n$, 
where $K_1=\bar{K}_1$ is the singleton graph and $P_n$
is the path on $n\ge1$ vertices,
we show that $\mathrm{QEC}(K_1+P_n)=-\Tilde{\alpha}_n-2$,
where $\Tilde{\alpha}_n$ is the minimal zero of 
a new polynomial $\Phi_n(x)$ related to
Chebyshev polynomials of the second kind.
Moreover, for an even $n$ we have
$\Tilde{\alpha}_n=\min\mathrm{ev}(A_n)$,
where the right-hand side is the minimal eigenvalue of 
the adjacency matrix $A_n$ of $P_n$.
For an odd $n$ we show that
$\min\mathrm{ev}(A_{n+1})\le\Tilde{\alpha}_n
<\min\mathrm{ev}(A_n)$.
\end{quote}

\begin{quote}
\textbf{Key words}\enspace
Chebyshev polynomial,
distance matrix,
fan graph,
graph join,
quadratic embedding constant
\end{quote}

\begin{quote}
\textbf{MSC}\enspace
primary:05C50.  \,\,  secondary:05C12, 05C76, 33C47.
\end{quote}

\begin{quote}
{\bfseries Acknowledgements} \enspace 
NO was supported by the JSPS Grant-in-Aid for Scientific Research 23K03126.
\end{quote}

\section{Introduction}
\label{01sec:Introduction}

Let $G=(V,E)$ be a finite connected graph with $|V|=n\ge2$
and $D=[d(i,j)]_{i,j\in V}$ the distance matrix of $G$,
where $d(i,j)$ stands for the standard graph 
distance between two vertices $i,j\in V$.
The \textit{quadratic embedding constant} 
(\textit{QE constant} for short)
of $G$ is defined by
\[
\mathrm{QEC}(G)
=\max\{\langle f,Df \rangle\,;\, f\in C(V), \,
\langle f,f \rangle=1, \, \langle \1,f \rangle=0\},
\]
where $C(V)$ is the space of real column vectors $f=[f_i]_{i\in V}$,
$\bm{1}\in C(V)$ the column vector whose entries are all one,
and $\langle\cdot,\cdot\rangle$ the canonical inner product.
The QE constant was first introduced in \cite{Obata2017,Obata-Zakiyyah2018}
for the quantitative approach to quadratic embedding of graphs
in Euclidean spaces,
which is a basic concept in Euclidean distance geometry, see e.g., 
\cite{Balaji-Bapat2007,Jaklic-Modic2013,Jaklic-Modic2014,Liberti-Lavor-Maculan-Mucherino2014}.
In fact, it follows from Schoenberg's theorem
\cite{Schoenberg1935,Schoenberg1938} that
a graph $G=(V,E)$ admits a quadratic embedding,
i.e., there exists a map $\varphi:V\rightarrow \mathbb{R}^N$
such that
\[
\|\varphi(i)-\varphi(j)\|^2=d(i,j),
\qquad i,j\in V,
\]
if and only if $\mathrm{QEC}(G)\le0$.
In recent years, there is growing interest in
the QE constant as a new numeric invariant of graphs,
see e.g., \cite{Baskoro-Obata2021,Choudhury-Nandi2023,
Irawan-Sugeng2021,Mlotkowski2022,MO-2018,Obata2023a,
Obata2023b}.

The aim of this paper is to obtain explicitly
the QE constant of a fan graph $K_1+P_n$,
that is the graph join of the singleton graph $K_1$
and the path $P_n$ with $n\ge1$.
The fan graphs form a basic family of graphs
and have been studied in various contexts,
for example, 
vertex coloring \cite{Falcon-Venkatachalam-Gowri-Nandini2021},
edge coloring \cite{Fitriani-Sugeng-Hariadi2018},
chromatic polynomials \cite{Maulana-Wijaya-Santoso2018},
inversion of the distance matrices \cite{Hao-Li-Zhang2022},
graph spectrum \cite{Liu-Yuan-Das2020} and so forth.
Although a fan graph has a relatively simple structure
at a first glance,
no reasonable formula for $\mathrm{QEC}(K_1+P_n)$ has been
obtained.
In the recent paper \cite{Lou-Obata-Huang2022} we derived a general
formula for $\mathrm{QEC}(G_1+G_2)$,
where both $G_1$ and $G_2$ are regular graphs;
however, that formula or its modification
is not applicable to a fan graph $K_1+P_n$.
The difficulty for $\mathrm{QEC}(K_1+P_n)$ might be expected 
also from the tough computation of $\mathrm{QEC}(P_n)$
performed in \cite{Mlotkowski2022}.

In this paper, going back to the definition of QE constant,
we first obtain a reasonably simple description for
$\mathrm{QEC}(\bar{K}_m+G)$, 
where $\bar{K}_m$ is the empty graph on $m\ge1$ vertices
and $G$ is an arbitrary graph (Theorem \ref{03thm:general formula}).
Then, using this new formula,
we obtain the QE constant of a fan graph $K_1+P_n$ as
\[
\mathrm{QEC}(K_1+P_n)=-\Tilde{\alpha}_n-2,
\qquad n\ge1,
\]
where $\Tilde{\alpha}_n$ is the minimal zero 
of the polynomial $\Phi_n(x)$ defined by
\[
\Phi_n(x)
=((n+1)x^2-6x-4n) \Tilde{U}_n(x)
+2(x+2)\Tilde{U}_{n-1}(x)
+2(x+2),
\]
where $\Tilde{U}_n(x)$ is 
the ``compressed'' Chebyshev polynomial of the second kind
defined by $\Tilde{U}_n(x)=U_n(x/2)=x^n+\dotsb$.
Moreover, for an even $n\ge2$ we have
\[
\Tilde\alpha_n=\min\mathrm{ev}(A_n)
=2\cos\frac{n\pi}{n+1}
=-2\cos\frac{\pi}{n+1},
\]
where $\mathrm{ev}(A_n)$ stands for the set of
eigenvalues of the adjacency matrix $A_n$ of the path $P_n$,
and for an odd $n\ge1$ we have
\[
-2\cos\frac{\pi}{n+2}
=\min\mathrm{ev}(A_{n+1})\le\Tilde{\alpha}_n<\min\mathrm{ev}(A_n)
=-2\cos\frac{\pi}{n+1}.
\]
The precise statement appears
in Theorem \ref{04eqn:main formula for a fan graph}.
It is noteworthy that the polynomial $\Phi_n(x)$ has a common factor 
with the Chebyshev polynomial $\Tilde{U}_n(x)$ 
in such a way that 
\begin{align*}
\Phi_n(x) &=(x-2)^2 \cdot \Tilde{U}^{\rm{e}}_n(x) \cdot R_n(x), \\
\Tilde{U}_n(x) &=\Tilde{U}^{\rm{e}}_n(x) \cdot \Tilde{U}^{\rm{o}}_n(x),
\end{align*}
where $\Tilde{U}^{\rm{e}}_n(x)$ and $\Tilde{U}^{\rm{o}}_n(x)$ are
monic polynomials with integer coefficients,
informally called ``partial Chebyshev polynomials,'' 
and $R_n(x)$ is another polynomial with integer coefficients,
see Appendix B.
The polynomials $\Phi_n(x)$ as well as
the partial Chebyshev polynomials
$\Tilde{U}^{\rm{e}}_n(x)$ and $\Tilde{U}^{\rm{o}}_n(x)$
seem to be new in literature
and worth exploring further.

This paper is organized as follows.
In Section 2 we prepare basic notions and 
preliminary results on general graph joins.
In Section 3 we derive a formula for $\mathrm{QEC}(\bar{K}_m+G)$,
where $G$ is an arbitrary graph.
The formula is stated in Theorem \ref{03thm:general formula}
and some examples are given for illustrating how our formula works.
In Section 4 we prove the main result on $\mathrm{QEC}(K_1+P_n)$,
see Theorem \ref{04eqn:main formula for a fan graph}.
In Appendix A we collect some useful facts on the three-term
recurrence equation with zero-boundary condition.
In Appendix B we give some interesting properties of
the polynomials $\Phi_n(x)$ which suggest a further study.

\section{Graph Joins in General}

Let $G_1=(V_1,E_1)$ and $G_2=(V_2,E_2)$ be two 
finite (not necessarily connected) graphs that are disjoint,
i.e., $V_1\cap V_2=\emptyset$.
For simplicity we set $m=|V_1|\ge 1$ and $n=|V_2|\ge1$.
The \textit{join} of $G_1$ and $G_2$,
denoted by $G_1+G_2$, is a graph with vertex set $V$ 
and edge set $E$ given by
\begin{align*}
V &=V_1\cup V_2, \\
E &=E_1\cup E_2\cup \{\{x,y\}\,;\, x\in V_1,\,\, y\in V_2\}.
\end{align*}
The graph join $G_1+G_2$ becomes
a finite connected graph even if $G_1$ or $G_2$ is not connected.

The distance matrix of the graph join $G_1+G_2$ is described
in terms of the adjacency matrices of $G_1$ and $G_2$,
see also \cite{Lou-Obata-Huang2022,Obata2017}.
Let $A_1$ and $A_2$ be the adjacency matrices of
$G_1$ and $G_2$, respectively.
These are $m\times m$ and $n\times n$ matrices, respectively.
In a natural manner the adjacency matrix of $G_1+G_2$ becomes
a block matrix of the form:
\begin{equation}\label{eqn:A of join}
A=
\begin{bmatrix}
A_1 & J \\
J & A_2 
\end{bmatrix},
\end{equation}
where $J$ is the matrix whose entries are all one
(this symbol is used without explicitly showing its size).
Then the distance matrix of $G_1+G_2$ is given by
\begin{equation}\label{2eqn:distance matrix of join}
D=2J-2I-A
=\begin{bmatrix}
  2J-2I-A_1 & J \\
   J  & 2J-2I-A_2
\end{bmatrix},
\end{equation}
where $I$ is the identity matrix
(its size is understood in the context).
We see from definition of the QE constant (Section \ref{01sec:Introduction})
that $\mathrm{QEC}(G_1+G_2)$ is the conditional maximum of
\begin{align}
\psi(f,g)
&=\left\langle \begin{bmatrix} f \\ g \end{bmatrix},
D \begin{bmatrix} f \\ g \end{bmatrix} \right\rangle 
\nonumber \\
&=\langle f, (2J-2I-A_1)f\rangle 
  +2\langle f, Jg\rangle
  +\langle g, (2J-2I-A_2)g\rangle
\label{02eqn:psi(f,g)}
\end{align}
subject to the constraints
$\langle f,f\rangle+\langle g,g\rangle =1$
and $\langle \1,f\rangle+\langle \1,g\rangle =0$,
where $f\in C(V_1)\cong \mathbb{R}^m$ 
and $g\in C(V_2)\cong \mathbb{R}^n$.

A general formula for QE constant of a graph is known
\cite[Proposition 4.1]{Obata-Zakiyyah2018}.
Following the original idea we will write down
a formula for $\mathrm{QEC}(G_1+G_2)$.
Upon applying the method of Lagrange's multipliers,
we set 
\begin{equation}\label{02eqn:varphi}
\varphi(\lambda,\mu,f,g)
=\psi(f,g)-\lambda(\langle f,f\rangle+\langle g,g\rangle-1)
 -\mu(\langle \bm{1},f\rangle+\langle \bm{1},g\rangle).
\end{equation}
Let $\mathcal{S}$ be the set of stationary points, namely, 
solutions $(\lambda,\mu,f,g)\in\mathbb{R}\times\mathbb{R}
\times \mathbb{R}^m\times \mathbb{R}^n$ to the system of equations:
\begin{equation}\label{02eqn:Lagrange}
\frac{\partial\varphi}{\partial f_1}
=\dots
=\frac{\partial\varphi}{\partial f_m}
=\frac{\partial\varphi}{\partial g_1}
=\dots
=\frac{\partial\varphi}{\partial g_n}
=\frac{\partial\varphi}{\partial \lambda}
=\frac{\partial\varphi}{\partial \mu}=0,
\end{equation}
where $f=[f_i]_{1\le i\le m}$ and $g=[g_j]_{1\le j\le n}$.
It follows from general theory
that the conditional maximum of $\psi(f,g)$ under consideration
is attained at a stationary point in $\mathcal{S}$.
On the other hand, we see by simple algebra that
$\psi(f,g)=\lambda$ for any $(\lambda,\mu,f,g)\in\mathcal{S}$.
Let $\lambda(\mathcal{S})$ be the set of $\lambda\in\mathbb{R}$
appearing in $\mathcal{S}$.
To be precise,
$\lambda(\mathcal{S})$ is the set of $\lambda\in\mathbb{R}$
for which there exist $\mu\in\mathbb{R}$,
$f\in \mathbb{R}^m$ and
$g\in \mathbb{R}^n$ such that
$(\lambda,\mu,f,g)\in\mathcal{S}$.
With these notations we come to the following formula:
\begin{equation}\label{02eqn:QEC(G) first step}
\mathrm{QEC}(G_1+G_2)
=\max\lambda(\mathcal{S}).
\end{equation}

Our next step is to obtain a concise description of $\mathcal{S}$.
We first show that equations \eqref{02eqn:Lagrange} are
equivalent to the following system of equations:
\begin{align}
&(A_1-J+\lambda I+2I)f =-\frac{\mu}{2}\1,
\label{02eqn:eqn-S(1)} \\
&(A_2-J+\lambda I+2I)g =-\frac{\mu}{2}\1,
\label{02eqn:eqn-S(2)} \\
&\langle f,f\rangle+\langle g,g\rangle =1,
\label{02eqn:eqn-S(3)} \\
&\langle \1,f\rangle+\langle \1,g\rangle =0.
\label{02eqn:eqn-S(4)}
\end{align}
Obviously, $\partial\varphi/\partial \lambda=0$ and 
$\partial\varphi/\partial \mu=0$ 
in \eqref{02eqn:Lagrange}
are equivalent to \eqref{02eqn:eqn-S(3)} and \eqref{02eqn:eqn-S(4)},
respectively.
We examine the equation
$\partial\varphi/\partial f_i=0$ in \eqref{02eqn:Lagrange}.
Letting $e_i\in C(V_1)\cong \mathbb{R}^m$ be the canonical basis
and taking \eqref{02eqn:eqn-S(4)} into account,
we obtain
\begin{align*}
\frac{\partial\varphi}{\partial f_i}
&=2\langle e_i, (2J-2I-A_1)f\rangle 
  +2\langle e_i, Jg\rangle -2\lambda\langle e_i, f\rangle -\mu \\
&=4\langle \bm{1},f\rangle
 -2\langle e_i, (2I+A_1)f\rangle 
  +2\langle \bm{1},g\rangle 
  -2\lambda\langle e_i, f\rangle 
  -\langle e_i,\mu\bm{1}\rangle \\
&=2\langle \bm{1},f\rangle
 -2\langle e_i, (2I+A_1)f\rangle 
 -2\lambda\langle e_i, f\rangle 
 -\langle e_i,\mu\bm{1}\rangle \\
&=2\langle e_i, (J-2I-A_1-\lambda I)f\rangle 
 -\langle e_i,\mu\bm{1}\rangle.
\end{align*}
Hence, under condition \eqref{02eqn:eqn-S(4)},
the equation $\partial \varphi/\partial f_i=0$ is equivalent to
\[
\langle e_i, (A_1-J+\lambda I+2I)f\rangle 
=\bigg\langle e_i,-\frac{\mu}{2}\,\bm{1}\bigg\rangle.
\]
Thus, the equations 
$\partial\varphi/\partial f_1=\dots=\partial\varphi/\partial f_m=0$
are consolidated into a single equation as in \eqref{02eqn:eqn-S(1)}.
Similarly, \eqref{02eqn:eqn-S(2)} is equivalent to
the equations $\partial\varphi/\partial g_1
=\dots=\partial\varphi/\partial g_n=0$
under condition \eqref{02eqn:eqn-S(4)}.
Consequently, the set $\mathcal{S}$ of stationary points of 
$\varphi(\lambda,\mu,f,g)$ coincides with the set
of $(\lambda,\mu,f,g)\in \mathbb{R}\times\mathbb{R}\times 
\mathbb{R}^m\times \mathbb{R}^n$ satisfying 
\eqref{02eqn:eqn-S(1)}--\eqref{02eqn:eqn-S(4)}.

We now introduce a new variable $\alpha$ by
\begin{equation}\label{02eqn:lambda and alpha}
\alpha=-\lambda-2
\,\,\,\text{or equivalently,}\,\,\,
\lambda=-\alpha-2.
\end{equation}
Then equations \eqref{02eqn:eqn-S(1)}--\eqref{02eqn:eqn-S(4)}
are rewritten as
\begin{align}
&(A_1-J-\alpha I)f =-\frac{\mu}{2}\1,
\label{02eqn:eqn-Sa(1)} \\
&(A_2-J-\alpha I)g =-\frac{\mu}{2}\1,
\label{02eqn:eqn-Sa(2)} \\
&\langle f,f\rangle+\langle g,g\rangle =1,
\label{02eqn:eqn-Sa(3)} \\
&\langle \1,f\rangle+\langle \1,g\rangle =0.
\label{02eqn:eqn-Sa(4)}
\end{align}
Let $\Tilde{\alpha}$ be the minimal $\alpha\in\mathbb{R}$ 
appearing in the solutions $(\alpha,\mu,f,g)$ to the
system of equations \eqref{02eqn:eqn-Sa(1)}--\eqref{02eqn:eqn-Sa(4)}.
Then $\lambda=-\Tilde{\alpha}-2$ coincides with
the maximal $\lambda\in\mathbb{R}$ appearing in $\mathcal{S}$,
which gives 
$\mathrm{QEC}(G_1+G_2)$ by \eqref{02eqn:QEC(G) first step}.

The above argument is summarized in the following statement.

\begin{theorem}\label{02thm:formula for QEC of join}
Let $G_1=(V_1,E_1)$ and $G_2=(V_2,E_2)$ be
(not necessarily connected) disjoint graphs and
$A_1$ and $A_2$ their adjacency matrices, respectively.
Let $\tilde{\alpha}$ be the minimal $\alpha\in\mathbb{R}$ appearing
in the solutions $(\alpha,\mu,f,g)
\in\mathbb{R}\times\mathbb{R}\times \mathbb{R}^m\times 
\mathbb{R}^n$ to the system of 
equations \eqref{02eqn:eqn-Sa(1)}--\eqref{02eqn:eqn-Sa(4)}.
Then, the QE constant of the graph join $G_1+G_2$ is given by
\begin{equation}\label{02eqn:main formula for QEC(G1+G2)}
\mathrm{QEC}(G_1+G_2)
=-\Tilde{\alpha}-2.
\end{equation}
\end{theorem}

\begin{remark}
\normalfont
An alternative or intermediate form of
\eqref{02eqn:main formula for QEC(G1+G2)}
is known \cite{Lou-Obata-Huang2022,Obata2017}.
Moreover, if both $G_1$ and $G_2$ are regular graphs,
the formula \eqref{02eqn:main formula for QEC(G1+G2)}
becomes more explicit in terms of the minimal eigenvalues 
of the adjacency matrices of $G_1$ and $G_2$,
see \cite[Theorem 3.1]{Lou-Obata-Huang2022}.
The purpose of this paper is to go slightly beyond this limitation.
\end{remark}

\section{Graph Join $\bar{K}_m+G$}

Let $G=(V,E)$ be a (not necessarily connected) graph with $|V|=n\ge1$
and $\bar{K}_m$ the empty graph on $m\ge1$ vertices.
We will derive a formula for $\mathrm{QEC}(\bar{K}_m+G)$.
Upon applying Theorem \ref{02thm:formula for QEC of join} 
we set $G_1=\bar{K}_m$, $A_1=0$, $G_2=G$ and $A_2=A$,
where $A$ is the adjacency matrix of $G$.
Then equations \eqref{02eqn:eqn-Sa(1)}--\eqref{02eqn:eqn-Sa(4)} become
\begin{align}
&(J+\alpha I)f=\frac{\mu}{2}\,\bm{1},
\label{03eqn:eqn-S(1)} \\
&(A-J-\alpha I)g=-\frac{\mu}{2}\,\1,
\label{03eqn:eqn-S(2)} \\
&\langle f,f\rangle+\langle g,g\rangle=1,
\label{03eqn:eqn-S(3)} \\
&\langle\bm{1},f\rangle+\langle \bm{1},g\rangle=0,
\label{03eqn:eqn-S(4)}
\end{align}
respectively.
Then the following assertion is a 
direct consequence of Theorem \ref{02thm:formula for QEC of join}.

\begin{proposition}\label{03prop:starting statement}
Notations and assumptions being as above,
let $\Lambda$ be the set of $\alpha\in\mathbb{R}$
appearing in the solutions
$(\alpha,\mu,f,g)\in\mathbb{R}\times
\mathbb{R}\times\mathbb{R}^m\times\mathbb{R}^n$
to equations \eqref{03eqn:eqn-S(1)}--\eqref{03eqn:eqn-S(4)}
and set
\begin{equation}\label{03eqn:min alpha (Lambda)}
\Tilde{\alpha}=\min\Lambda.
\end{equation}
Then we have
\begin{equation}\label{03eqn:main formula for G}
\mathrm{QEC}(\bar{K}_m+G)
=-\Tilde{\alpha}-2.
\end{equation}
\end{proposition}

It is known \cite{Baskoro-Obata2021} that
$\mathrm{QEC}(G)\ge-1$ for any graph $G$
and the equality occurs only when $G$ is a complete graph.
If $m=1$ (then $\bar{K}_m=\bar{K}_1=K_1$ is the singleton graph)
and $G=K_n$ is a complete graph, we have
\[
\mathrm{QEC}(\bar{K}_1+K_n)=\mathrm{QEC}(K_{n+1})=-1.
\]
Otherwise, we have $\mathrm{QEC}(\bar{K}_m+G)>-1$,
and hence $\Tilde{\alpha}<-1$ by \eqref{03eqn:main formula for G}.
It is then noted that
\eqref{03eqn:min alpha (Lambda)} becomes 
\begin{equation}\label{03eqn:min alpha (Lambda)<-1}
\Tilde{\alpha}
=\min\Lambda
=\min\Lambda\cap(-\infty,-1).
\end{equation}
From now on we focus on the case where
the graph join $\bar{K}_m+G$ is not a complete graph.
Our main task is obtain a reasonably simple
description of $\Lambda\cap(-\infty,-1)$.

We start with \eqref{03eqn:eqn-S(1)}.
The following general result is useful,
of which the proof is elementary.

\begin{lemma}\label{03lem:linear equation with J}
Let $m\ge1$ and $J$ the $m\times m$ matrix with all entries being one.
Then all 
$(\alpha,\mu,f)\in\mathbb{R}\times\mathbb{R}\times \mathbb{R}^m$
satisfying the equation:
\[
(J+\alpha I)f=\frac{\mu}{2}\,\bm{1}
\]
are determined as follows:
\begin{enumerate}
\item[\upshape (i)] $\alpha=0$,
$\mu\in\mathbb{R}$ and $f\in\mathbb{R}^m$ are arbitrary
with $\mu=2\langle\bm{1},f\rangle$;
\item[\upshape (ii)] $\alpha=-m$, 
$\mu=0$ and $f=c\bm{1}$ with arbitrary $c\in\mathbb{R}$;
\item[\upshape (iii)] $\alpha\in\mathbb{R}\backslash\{0,-m\}$,
$\mu\in\mathbb{R}$ and $f\in\mathbb{R}^m$ are arbitrary with
\begin{equation}\label{03eqn: condition (1)}
f=\frac{1}{\alpha+m}\cdot\frac{\mu}{2}\,\bm{1}.
\end{equation}
\end{enumerate}
\end{lemma}

We will determine $\Lambda\cap(-\infty,-1)$
according to the three cases 
in Lemma \ref{03lem:linear equation with J}.
First the case (i) is discarded.
The case (ii) is also out of consideration for $m=1$.
For $m\ge2$ we have the following assertion.

\begin{lemma}
Let $m\ge2$.
Then $\alpha=-m$ appears in the solutions to
\eqref{03eqn:eqn-S(1)}--\eqref{03eqn:eqn-S(4)}
if and only if $m$ is an eigenvalue of $J-A$,
that is, $m\in\mathrm{ev}(J-A)$.
\end{lemma}

\begin{proof}
We put $\alpha=-m$, $\mu=0$ and
$f=c\bm{1}$ with $c\in\mathbb{R}$
in equations \eqref{03eqn:eqn-S(1)}--\eqref{03eqn:eqn-S(4)}.
Then \eqref{03eqn:eqn-S(1)} is satisfied
by Lemma \ref{03lem:linear equation with J} (ii)
and the rest of the equations reduce to the following:
\begin{align}
&(A-J+mI)g=0,
\label{03eqn:eqn-S1(2)} \\
&c^2m+\langle g,g\rangle=1,
\label{03eqn:eqn-S1(3)} \\
&cm+\langle \bm{1},g\rangle=0,
\label{03eqn:eqn-S1(4)}
\end{align}
respectively.
If $g=0$, we have $c=0$ by \eqref{03eqn:eqn-S1(4)}
and hence \eqref{03eqn:eqn-S1(3)} fails.
Therefore, $g\neq0$ and
we see from \eqref{03eqn:eqn-S1(2)} that 
$m$ is an eigenvalue of $J-A$.

Conversely, suppose that $m\in\mathrm{ev}(J-A)$
and take an associated eigenvector $g_0\neq0$.
Setting $g=\gamma g_0$ and $f=c\bm{1}$ with $\gamma,c\in\mathbb{R}$,
we specify $c$ and $\gamma$ so as to fulfill
\eqref{03eqn:eqn-S(3)} and \eqref{03eqn:eqn-S(4)}.
Then, $(\alpha=-m,\mu=0, f=c\bm{1},g=\gamma g_0)$ becomes a 
solution to
\eqref{03eqn:eqn-S(1)}--\eqref{03eqn:eqn-S(4)}.
\end{proof}

We put
\[
\Lambda_0=
\begin{cases}
\{-m\}, & \text{if $m\ge2$ and $m\in\mathrm{ev}(J-A)$}, \\
\emptyset, & \text{otherwise}, \\
\end{cases}
\]
where $A$ is the adjacency matrix of $G$.
It is noted that
\begin{equation}\label{03eqn:Lambda0}
\Lambda\cap(-\infty,-1)\cap \{-m\}=\Lambda_0,
\qquad m\ge1.
\end{equation}

We next consider the case (iii) 
in Lemma \ref{03lem:linear equation with J}.
Let $\alpha\in\mathbb{R}\backslash\{0,-m\}$,
$\mu\in\mathbb{R}$ and $f$ be related 
as in \eqref{03eqn: condition (1)}.
We see from \eqref{03eqn:eqn-S(4)} and \eqref{03eqn: condition (1)}
that
\[
Jg
=\langle\bm{1},g\rangle\bm{1}
=-\langle\bm{1},f \rangle \bm{1}
=-\frac{m}{\alpha+m}\, \frac{\mu}{2}\, \bm{1}.
\]
Hence, \eqref{03eqn:eqn-S(2)} is rewritten as
\begin{equation}\label{03eqn: condition (2)}
(A-\alpha I)g
=Jg-\frac{\mu}{2}\1
=-\frac{\alpha+2m}{\alpha+m}\, \frac{\mu}{2} \bm{1}.
\end{equation}
Thus, $\alpha\in\mathbb{R}\backslash\{0,-m\}$ appears 
in the solutions to \eqref{03eqn:eqn-S(1)}--\eqref{03eqn:eqn-S(4)}
if and only if it appears 
in the solutions to the alternative system of equations
\eqref{03eqn: condition (1)},
\eqref{03eqn: condition (2)},
\eqref{03eqn:eqn-S(3)} and \eqref{03eqn:eqn-S(4)}.
We will determine such $\alpha$'s according to the
partition:
\begin{equation}\label{03eqn:partition}
\mathbb{R}\backslash\{0,-m\}
=\Omega_1\cup \Omega_2\cup \Omega_3\,,
\end{equation}
where
\[
\Omega_1=\mathbb{R}\backslash(\mathrm{ev}(A)\cup\{0,-m,-2m\}),
\quad
\Omega_2=\{-2m\},
\quad
\Omega_3=\mathrm{ev}(A)\backslash\{0,-m,-2m\}.
\]

\begin{lemma}\label{04lem:not in ev(A)}
Let $\alpha\in\mathbb{R}\backslash(\mathrm{ev}(A)\cup\{0,-m,-2m\})$.
Then $\alpha$ appears in the solutions to
\eqref{03eqn:eqn-S(1)}--\eqref{03eqn:eqn-S(4)},
that is, $\alpha\in\Lambda$ if and only if
\begin{equation}\label{03eqn:equation Lambda1}
(\alpha+2m)\langle \bm{1}, (A-\alpha I)^{-1}\bm{1}\rangle-m=0.
\end{equation}
\end{lemma}

\begin{proof}
Suppose that $\alpha$ appears in the solutions to
\eqref{03eqn:eqn-S(1)}--\eqref{03eqn:eqn-S(4)}.
Then there exists $(\alpha,\mu,f,g)$ satisfying
\eqref{03eqn: condition (1)},
\eqref{03eqn: condition (2)},
\eqref{03eqn:eqn-S(3)} and \eqref{03eqn:eqn-S(4)}.
From \eqref{03eqn: condition (2)} we obtain
\begin{equation}\label{03eqn:g by inverse}
g=-\frac{\alpha+2m}{\alpha+m}\,\frac{\mu}{2}\,
(A-\alpha I)^{-1}\bm{1}.
\end{equation}
Inserting \eqref{03eqn: condition (1)} and 
\eqref{03eqn:g by inverse} into \eqref{03eqn:eqn-S(4)},
we obtain
\begin{equation}\label{03eqn:in proof (31)}
\frac{m}{\alpha+m}\,\frac{\mu}{2}
-\frac{\alpha+2m}{\alpha+m}\,\frac{\mu}{2}
\langle\bm{1},(A-\alpha I)^{-1}\bm{1}\rangle=0.
\end{equation}
If $\mu=0$, then $f=0$ by \eqref{03eqn: condition (1)}
and $g=0$ by \eqref{03eqn:g by inverse},
which do not fulfill \eqref{03eqn:eqn-S(3)}.
Hence $\mu\neq0$ and \eqref{03eqn:equation Lambda1}
follows from \eqref{03eqn:in proof (31)}.

Conversely, suppose that $\alpha$ satisfies 
\eqref{03eqn:equation Lambda1}.
We define $f$ and $g$ by \eqref{03eqn: condition (1)}
and \eqref{03eqn:g by inverse}, respectively.
Then, specifying $\mu$ so as to fulfill \eqref{03eqn:eqn-S(3)},
we obtain a solution $(\alpha,\mu,f,g)$ to
\eqref{03eqn:eqn-S(1)}--\eqref{03eqn:eqn-S(4)}
and hence $\alpha\in\Lambda$.
\end{proof}

Let $\Tilde{\Lambda}_1$ be 
the set of solutions to \eqref{03eqn:equation Lambda1}
and put  
\[
\Lambda_1=\Tilde{\Lambda}_1
\backslash(\mathrm{ev}(A)\cup\{0,-m,-2m\}).
\]
Then we obtain
\begin{equation}\label{03eqn:Lambda1}
\Lambda\backslash(\mathrm{ev}(A)\cup\{0,-m,-2m\})=\Lambda_1.
\end{equation}
Note that $\Lambda_1$ may contain $\alpha\ge-1$.

\begin{lemma}\label{03lem:Lambda2}
$\alpha=-2m$ appears in the solution to
\eqref{03eqn:eqn-S(1)}--\eqref{03eqn:eqn-S(4)},
that is, $-2m\in\Lambda$ if and only if $-2m\in \mathrm{ev}(A)$.
\end{lemma}

\begin{proof}
Setting $\alpha=-2m$, 
\eqref{03eqn: condition (1)} becomes
\begin{equation}\label{03eqn:in proof 3.4(1)}
f=-\frac{1}{m}\,\frac{\mu}{2}\,\bm{1},
\end{equation}
and \eqref{03eqn: condition (2)} becomes
\begin{equation}\label{03eqn:in proof 3.4(2)}
(A+2mI)g=0.
\end{equation}
If $-2m\not\in\mathrm{ev}(A)$, from 
\eqref{03eqn:in proof 3.4(2)} we obtain $g=0$.
Then by \eqref{03eqn:eqn-S(4)} we have $\langle\bm{1},f\rangle=0$
and by \eqref{03eqn:in proof 3.4(1)} we have $\mu=0$ and $f=0$.
These do not fulfill \eqref{03eqn:eqn-S(3)}.
Hence it follows that $-2m\in\mathrm{ev}(A)$.

Conversely, suppose that $\alpha=-2m\in\mathrm{ev}(A)$.
We choose an eigenvector $g_0\neq 0$ such that $Ag_0=-2mg_0$
and set $g=\gamma g_0$ with $\gamma\in\mathbb{R}$.
Then \eqref{03eqn: condition (2)} is satisfied.
On the other hand, \eqref{03eqn:eqn-S(3)}
and \eqref{03eqn:eqn-S(4)} become
\begin{align}
& \frac{1}{m}\bigg(\frac{\mu}{2}\bigg)^2
 +\gamma^2\langle g_0,g_0 \rangle =1,
\label{03eqn: condition (23)} \\
& -\frac{\mu}{2}+\gamma \langle \bm{1},g_0 \rangle =0,
\label{03eqn: condition (24)}
\end{align}
respectively.
We may choose $\gamma$ and $\mu$ satisfying 
\eqref{03eqn: condition (23)}
and \eqref{03eqn: condition (24)},
and define $f$ by \eqref{03eqn:in proof 3.4(1)}.
Thus $(\alpha=-2m, \mu,f,g=\gamma g_0)$ becomes a solution to
\eqref{03eqn:eqn-S(1)}--\eqref{03eqn:eqn-S(4)}.
\end{proof}

We put
\[
\Lambda_2=
\begin{cases}
\{-2m\}, & \text{if $-2m\in \mathrm{ev}(A)$}, \\
\emptyset, & \text{if $-2m\not\in \mathrm{ev}(A)$}.
\end{cases}
\]
Then we have
\begin{equation}\label{03eqn:Lambda2}
\Lambda_2=\Lambda\cap\{-2m\}.
\end{equation}

\begin{lemma}\label{03lem:Lambda3}
Let $\alpha\in \mathrm{ev}(A)\backslash\{0,-m,-2m\}$.
Then $\alpha$ appears in the solution to
\eqref{03eqn:eqn-S(1)}--\eqref{03eqn:eqn-S(4)},
that is, $\alpha\in\Lambda$ 
if and only if it admits an eigenvector $g$ such that 
\[
Ag=\alpha g,
\qquad \langle g,g\rangle=1,
\qquad \langle \bm{1},g\rangle=0.
\]
\end{lemma}

\begin{proof}
Assume first that there exists an eigenvector $g_0$ such that
\begin{equation}\label{03eqn:existence of g_0}
Ag_0=\alpha g_0,
\qquad \langle g_0, g_0\rangle=1,
\qquad \langle \bm{1},g_0\rangle=0.
\end{equation}
Then $(\alpha, 0,0,g_0)$ fulfills
\eqref{03eqn:eqn-S(1)}--\eqref{03eqn:eqn-S(4)}.

We next assume there exists no vector $g_0$ satisfying 
\eqref{03eqn:existence of g_0}.
Suppose first that $\mu=0$.
Then equations \eqref{03eqn: condition (1)},
\eqref{03eqn: condition (2)},
\eqref{03eqn:eqn-S(3)} and \eqref{03eqn:eqn-S(4)}
respectively become
\[
f=0,
\qquad Ag=\alpha g,
\qquad \langle g,g\rangle=1,
\qquad \langle \bm{1},g\rangle=0,
\]
which contradict to the assumption of the present case.
Next suppose that $\mu\neq0$.
Then the right-hand side of
\eqref{03eqn: condition (2)} is not zero and it has a solution
if and only if
\begin{equation}\label{03eqn:ranks}
\mathrm{rank}(A-\alpha I)
=\mathrm{rank} [A-\alpha I, \bm{1}],
\end{equation}
where $[A-\alpha I, \bm{1}]$ is an $n\times (n+1)$ matrix 
obtained by adding a column vector $\bm{1}$.
For $1\le i\le n$ let $v_i$ be the $i$th row vector of
$[A-\alpha, \bm{1}]$, that is,
\[
v_i=[(A-\alpha I)_{i1},\dots, (A-\alpha I)_{in},1].
\]
We will show that $v_1,v_2,\dots,v_n$ are linearly independent.
In fact, suppose that
\[
\sum_{i=1}^n \xi_i v_i=0,
\qquad \xi_1,\dots,\xi_n\in\mathbb{R}.
\]
Comparing the $j$th components of the both sides,
we obtain
\begin{align}
\sum_{i=1}^n \xi_i (A-\alpha I)_{ij}&=0,
\qquad 1\le j\le n, 
\label{03eqn:in proof 8}\\
\sum_{i=1}^n \xi_i &=0,
\qquad j=n+1.
\label{03eqn:in proof 9}
\end{align}
Since $A-\alpha$ is a symmetric matrix, 
\eqref{03eqn:in proof 8} is equivalent to
\[
\sum_{i=1}^n (A-\alpha I)_{ji}\xi_i =0,
\]
namely,
\begin{equation}\label{03eqn:in proof 10}
(A-\alpha I)\xi=0,
\end{equation}
where $\xi=[\xi_1,\dots,\xi_n]^T$.
Similarly, from \eqref{03eqn:in proof 9} we obtain
\begin{equation}\label{03eqn:in proof 11}
\langle \bm{1},\xi\rangle=0.
\end{equation}
Since there exists no
non-zero vector satisfying 
\eqref{03eqn:in proof 10} and \eqref{03eqn:in proof 11}
by assumption of the present case,
we conclude that $\xi=0$ and
that $v_1,v_2,\dots,v_n$ are linearly independent.
Since they are the row vectors of the 
matrix $[A-\alpha I, \bm{1}]$, we have 
\[
\mathrm{rank}[A-\alpha I, \bm{1}]=n.
\]
On the other hand, since $\alpha$ is an eigenvalue of $A$, we have
\[
\mathrm{rank}(A-\alpha I)\le n-1.
\]
Hence \eqref{03eqn:ranks} fails
and \eqref{03eqn: condition (2)} has no solution.
\end{proof}

Let $\Lambda_3$ be the set of 
$\alpha\in \mathrm{ev}(A)\backslash\{0,-m,-2m\}$
which admits an eigenvector $g$ such that 
\[
Ag=\alpha g,
\qquad \langle g,g\rangle=1,
\qquad \langle \bm{1},g\rangle=0.
\]
Then we have
\begin{equation}\label{03eqn:Lambda3}
\Lambda_3=\Lambda\cap (\mathrm{ev}(A)\backslash\{0,-m,-2m\}).
\end{equation}
Note that $\Lambda_3$ may contain $\alpha\ge-1$.

Now we recall the partition $\mathbb{R}\backslash\{0,-m\}
=\Omega_1\cup\Omega_2\cup\Omega_3$ in \eqref{03eqn:partition}.
Since 
\[
\Lambda_1=\Lambda\cap \Omega_1,
\qquad \Lambda_2=\Lambda\cap \Omega_2,
\qquad \Lambda_3=\Lambda\cap \Omega_3,
\]
by \eqref{03eqn:Lambda1}, \eqref{03eqn:Lambda2} and \eqref{03eqn:Lambda3},
we have
\[
\Lambda\backslash\{0,-m\}=\Lambda_1\cup\Lambda_2\cup\Lambda_3\,.
\]
Combining \eqref{03eqn:Lambda0}, we come to the crucial partition
of $\Lambda\cap(-\infty,-1)$:
\begin{equation}
\Lambda\cap(-\infty,-1)
=(\Lambda_0\cup\Lambda_1\cup\Lambda_2\cup\Lambda_3)
\cap(-\infty,-1).
\end{equation}
Recall that $\Tilde{\alpha}=\min\Lambda$ by definition,
see Proposition \ref{03prop:starting statement}.
Since $\Tilde{\alpha}<-1$ by \eqref{03eqn:min alpha (Lambda)<-1},
we have
\begin{align}
\Tilde{\alpha}
&=\min\Lambda 
\nonumber \\
&=\min\, (\Lambda_0\cup\Lambda_1\cup\Lambda_2\cup\Lambda_3)\cap(-\infty,-1)
\nonumber \\
&=\min\, (\Lambda_0\cup\Lambda_1\cup\Lambda_2\cup\Lambda_3).
\label{03eqn:Tilde alpha final}
\end{align}
Thus, the formula in Proposition \ref{03prop:starting statement} can
be written down more explicitly in terms of
$\Lambda_0$, $\Lambda_1$, $\Lambda_2$, $\Lambda_3$.
The result is stated in the following theorem.

\begin{theorem}\label{03thm:general formula}
Let $G=(V,E)$ be a (not necessarily connected) graph with $|V|=n\ge1$
and $\bar{K}_m$ the empty graph on $m\ge1$ vertices.
Assume that $\bar{K}_m+G$ is not a complete graph.
Then we have
\[
\mathrm{QEC}(\bar{K}_m+G)=-\Tilde{\alpha}-2,
\]
where $\Tilde{\alpha}$ is
given by \eqref{03eqn:Tilde alpha final}.
\end{theorem}

The following examples will illustrate 
how Theorem \ref{03thm:general formula} works.

\begin{example}\label{03ex:diamond}
\normalfont
Let $G=K_2$ and $m=2$.
The graph join $\bar{K}_2+K_2$ is called a \textit{diamond}.
Let $A$ be the adjacency matrix of $K_2$.
Then
\begin{gather*}
A=\begin{bmatrix} 0 & 1 \\ 1 & 0 \end{bmatrix},
\qquad
|A-\alpha I|=\alpha^2-1,
\quad
\mathrm{ev}(A)=\{\pm1\},
\\
(A-\alpha I)^{-1}
=\frac{1}{|A-\alpha I|}
\begin{bmatrix}
-\alpha & -1 \\
-1 & -\alpha
\end{bmatrix},
\qquad
J-A=\begin{bmatrix} 1 & 0 \\ 0 & 1 \end{bmatrix}.
\end{gather*}
Since $\mathrm{ev}(J-A)=\{1\}$, we have $\Lambda_0=\emptyset$.
Summing up the entries of $(A-\alpha I)^{-1}$,
we obtain
\[
\langle \bm{1}, (A-\alpha I)^{-1}\bm{1}\rangle
=\frac{1}{|A-\alpha I|}(-2\alpha-2)
=-\frac{2}{\alpha-1}\,.
\]
Thus, equation \eqref{03eqn:equation Lambda1} becomes
\[
(\alpha+4)\,\frac{-2}{\alpha-1}-2=0,
\]
which has a unique solution $\alpha=-3/2$.
In view of $-3/2\not\in\mathrm{ev}(A)\cup\{0,-2,-4\}$
we obtain $\Lambda_1=\{-3/2\}$.
Since $-4\not\in\mathrm{ev}(A)$, we have $\Lambda_2=\emptyset$.
Finally, since $\mathrm{ev}(A)\backslash\{0,-2,-4\}=\{\pm1\}$,
we have $\Lambda_3\cap(-\infty,-1)=\emptyset$.
Consequently,
\[
\Tilde\alpha
=\min(\Lambda_0\cup\Lambda_1
 \cup\Lambda_2\cup\Lambda_3)\cap(-\infty,-1)
=\min\bigg\{-\frac32\bigg\}
=-\frac32,
\]
and hence
\[
\mathrm{QEC}(\bar{K}_2+K_{2})
=-\Tilde{\alpha}-2
=-\frac12.
\]
\end{example}

\begin{example}\label{03ex:C4}
\normalfont
Let $G=C_4$ and $m=1$. Then $K_1+C_4=W_4$ becomes a \textit{wheel}.
Obviously, $\Lambda_0=\emptyset$.
Let $A$ be the adjacency matrix $G$.
Then we have
\begin{gather*}
A=\begin{bmatrix}
0 & 1 & 0 & 1\\
1 & 0 & 1 & 0\\
0 & 1 & 0 & 1\\
1 & 0 & 1 & 0
\end{bmatrix},
\qquad
|A-\alpha I|=\alpha^2(\alpha^2-4),
\quad \mathrm{ev}(A)=\{0,\pm2\},
\\
(A-\alpha I)^{-1}
=\frac{1}{|A-\alpha I|}
\begin{bmatrix}
2\alpha -\alpha^3 & -\alpha^2 & -2\alpha & -\alpha^2 \\
-\alpha^2 & 2\alpha - \alpha^3 & -\alpha^2 & -2\alpha \\
-2\alpha & -\alpha^2 & 2 \alpha - \alpha^3 & -\alpha^2 \\
-\alpha^2 & -2\alpha & -\alpha^2 & 2\alpha-\alpha^3
\end{bmatrix}.
\end{gather*}
Summing up the entries of $(A-\alpha I)^{-1}$, we obtain
\[
\langle \bm{1}, (A-\alpha I)^{-1}\bm{1}\rangle
=\frac{1}{|A-\alpha I|}
 (-4\alpha^3-8\alpha^2)
=-\frac{4}{\alpha-2}\,.
\]
Thus, equation \eqref{03eqn:equation Lambda1} becomes
\[
(\alpha+2)\,\frac{-4}{\alpha-2}-1=0
\]
and we obtain $\alpha=-6/5$
which does not belong to $\mathrm{ev}(A)\cup\{0,-1,-2\}$.
Therefore, $\Lambda_1=\{-6/5\}$.
Next we have $\Lambda_2=\{-2\}$ since $-2\in\mathrm{ev}(A)$.
Finally, 
since $\mathrm{ev}(A)\backslash\{0,-1,-2\}=\{2\}$,
we have $\Lambda_3\cap(-\infty,-1)=\emptyset$.
Consequently,
\[
\Tilde\alpha
=\min(\Lambda_0\cup\Lambda_1
 \cup\Lambda_2\cup\Lambda_3)\cap(-\infty,-1)
=\min\bigg\{-\frac65,-2\bigg\}
=-2,
\]
and hence
\[
\mathrm{QEC}(K_1+C_4)=-\Tilde{\alpha}-2=0.
\]
The QE constant of a general wheel graph $W_n=K_1+C_n$ with $n\ge3$ 
is known \cite{Obata2017}.
\end{example}

\begin{example}\label{03ex:barK_m+P3}
\normalfont
Let $G=P_3$ and consider $\bar{K}_m+P_3$ with $m\ge1$.
Let $A$ be the adjacency matrix $P_3$.
Then we have
\begin{gather*}
A=\begin{bmatrix}
0 & 1 & 0 \\
1 & 0 & 1 \\
0 & 1 & 0
\end{bmatrix},
\qquad
|A-\alpha I|=-\alpha^3+2\alpha,
\quad 
\mathrm{ev}(A)=\{0,\pm\sqrt2\},
\\
(A-\alpha I)^{-1}
=\frac{1}{|A-\alpha I|}
\begin{bmatrix}
\alpha^2-1 & \alpha & 1 \\
\alpha &\alpha^2 & \alpha \\
1 & \alpha & \alpha^2-1
\end{bmatrix}
\\
J-A=\begin{bmatrix}
1 & 0 & 1 \\
0 & 1 & 0 \\
1 & 0 & 1
\end{bmatrix},
\qquad
\mathrm{ev}(J-A)=\{0,1,2\}.
\end{gather*}
It follows directly from definition that
\[
\Lambda_0=
\begin{cases}
\{-2\}, & \text{if $m=2$}, \\
\emptyset, & \text{otherwise}.
\end{cases}
\]
For $\Lambda_1$ we need an explicit form of
equation \eqref{03eqn:equation Lambda1}.
Summing up all entries of $(A-\alpha I)^{-1}$, we come to
\[
\langle \bm{1},(A-\alpha I)^{-1}\bm{1}\rangle
=\frac{3\alpha^2+4\alpha}{-\alpha^3+2\alpha}
=-\frac{3\alpha+4}{\alpha^2-2}\,.
\]
Then equation \eqref{03eqn:equation Lambda1} becomes
\[
-(\alpha+2m)\,\frac{3\alpha+4}{\alpha^2-2}-m=0,
\]
of which solutions are 
\[
\alpha=\alpha_{\pm}=\frac{-3m-2\pm\sqrt{3m^2-6m+4}}{m+3}\,.
\]
By definition $\Lambda_1$ consists of $\alpha_{\pm}$
satisfying $\alpha_{\pm}\not\in\mathrm{ev}(A)\cup\{0,-m,-2m\}$.
After simple observation we come to
\[
\Lambda_1=
\begin{cases}
\{\alpha_-=-3/2\}, & \text{if $m=1$}, \\
\{\alpha_+=-6/5\}, & \text{if $m=2$}, \\
\{\alpha_+,\alpha_-\}, & \text{if $m\ge3$}.
\end{cases}
\]
In view of $-2m\not\in\mathrm{ev}(A)$ we have $\Lambda_2=\emptyset$.
Since any eigenvector associated to
$\alpha\in\mathrm{ev}(A)\backslash\{0,-m,-2m\}
=\{\pm\sqrt2\}$ is not orthogonal to $\bm{1}$,
we have $\Lambda_3=\emptyset$.
Thus,
\[
\Lambda_0\cup\Lambda_1\cup\Lambda_2\cup\Lambda_3=
\begin{cases}
\{-3/2\}, & \text{if $m=1$}, \\
\{-2,-6/5\}, & \text{if $m=2$}, \\
\{\alpha_+,\alpha_-\}, & \text{if $m\ge3$}.
\end{cases}
\]
and $\Tilde\alpha
=\min \Lambda_0\cup\Lambda_1\cup\Lambda_2\cup\Lambda_3$
is given by
\[
\Tilde\alpha=
\begin{cases}
-3/2, & \text{if $m=1$}, \\
-2,& \text{if $m=2$}, \\
\alpha_-, & \text{if $m\ge3$}.
\end{cases}
\]
Since the values of $\alpha_-$ for $m=1$ and $m=2$
are respectively $-3/2$ and $-2$,
we obtain a unified expression:
\[
\Tilde\alpha=\alpha_-
=\frac{-3m-2-\sqrt{3m^2-6m+4}}{m+3}\,,
\qquad m\ge1.
\]
Consequently,
\[
\mathrm{QEC}(\bar{K}_m+P_3)
=-\Tilde{\alpha}-2
=\frac{m-4+\sqrt{3m^2-6m+4}}{m+3}\,,
\qquad m\ge1.
\]
Note that $\bar{K}_1+P_3$ is a diamond (Example \ref{03ex:diamond})
and $\bar{K}_2+P_3$ is a wheel $W_4$ (Example \ref{03ex:C4}).
These two graphs are of QE class.
It is noteworthy that $\bar{K}_m+P_3$ is of non-QE class
for $m\ge3$.
In fact, we have $\mathrm{QEC}(\bar{K}_m+P_3)>0$ for $m\ge3$.
\end{example}

\begin{remark}
\normalfont
Examples \ref{03ex:diamond}--\ref{03ex:barK_m+P3} show that
$\Lambda_0$, $\Lambda_1$ and $\Lambda_2$ are necessary to
determine the value $\Tilde{\alpha}$ while $\Lambda_3=\emptyset$.
In the next section we will see that $\Lambda_3\neq\emptyset$
occurs for $K_1+P_n$.
Thus, generally speaking, 
all of $\Lambda_0$, $\Lambda_1$, $\Lambda_2$ and $\Lambda_3$
are necessary to determine $\Tilde{\alpha}$ by
formula \eqref{03eqn:Tilde alpha final}.
\end{remark}

\begin{remark}\label{03rem:regular graph}
\normalfont
For a (not necessarily connected) regular graph $G$
we can derive a formula for $\mathrm{QEC}(K_1+G)$ 
as an application of Theorem \ref{03thm:general formula}.
In fact, if $G$ is a regular graph on $n$ vertices 
with degree $\kappa\ge1$, we have
\begin{equation}\label{03eqn:QEC(K1+G) with regular G}
\mathrm{QEC}(K_1+G)
=\max \left\{-\frac{\kappa+2}{n+1},\,
 -\min \mathrm{ev}(A)-2\right\}.
\end{equation}
The above formula follows also from
a more general formula \cite[Theorem 3.1]{Lou-Obata-Huang2022}
for $\mathrm{QEC}(G_1+G_2)$ where both $G_1$ and $G_2$
are regular graphs.
\end{remark}

\section{Fan Graph $K_1+P_n$}
\label{04sec:K1+Pn}

For $n\ge1$ let $P_n$ be the path graph on $n$ vertices.
The graph join $K_1+P_n$ is called a \textit{fan}.
The aim of this section is to determine $\mathrm{QEC}(K_1+P_n)$. 
For $n=1$ and $n=2$, the fan graph $K_1+P_n$ 
being respectively the complete graphs $K_2$ and $K_3$, 
we have 
\begin{equation}\label{04eqn:K_1+P_1 and P_2}
\mathrm{QEC}(K_1+P_1)=\mathrm{QEC}(K_1+P_2)=-1.
\end{equation}
For $\mathrm{QEC}(K_1+P_n)$ with $n\ge3$ 
the situation becomes rather complicated,
where the known formula or its modification
(see also Remark \ref{03rem:regular graph}) is not applicable.
On the other hand, we see easily that
\begin{equation}\label{04eqn:K_1+P_n in K_1+P_n+1}
\mathrm{QEC}(K_1+P_n)\le \mathrm{QEC}(K_1+P_{n+1}),
\qquad
n\ge1,
\end{equation}
as an immediate consequence from the general fact that
if $H$ is an isometrically embedded subgraph of a graph $G$, 
we have $\mathrm{QEC}(H)\le \mathrm{QEC}(G)$,
see \cite[Theorem 3.1]{Obata-Zakiyyah2018}.

We prepare some notations.
For $n\ge1$ let $A_n$ be the adjacency matrix of $P_n$ for which
we adopt the standard form:
\begin{equation}\label{04eqn:adjacency matrix of Pn}
A_n=\begin{bmatrix}
0 & 1 &  \\
1 & 0 & 1 & \\
  & 1 & 0 & 1 & \\
  &         & \ddots  & \ddots & \ddots \\
  &         &         & 1       & 0 & 1 \\
  &         &         &        & 1 & 0
\end{bmatrix}.
\end{equation}
We use the \textit{Chebyshev polynomial of the second kind}
defined by
\begin{equation}\label{04eqn:Chebyshev 2nd kind}
U_n(x)=\frac{\sin(n+1)\theta}{\sin\theta}\,,
\qquad x=\cos \theta,
\qquad n\ge0,
\end{equation}
and their normalization:
\[
\Tilde{U}_n(x)=U_n\bigg(\frac{x}{2}\bigg)=x^n+\dotsb,
\qquad n\ge0.
\]
It is well known that
the characteristic polynomial of $A_n$ is given by
\[
|xI-A_n|=\Tilde{U}_n(x),
\qquad n\ge1.
\]
With the help of the obvious relation:
\[
U_n\bigg(\cos\frac{l\pi}{n+1}\bigg)=0,
\qquad 1\le l\le n,
\]
the eigenvalues of $A_n$ are obtained as
\begin{equation}\label{04eqn:ev(A_n)}
\mathrm{ev}(A_n)
=\{\alpha_1>\alpha_2>\dots>\alpha_n\},
\end{equation}
where
\begin{equation}\label{04eqn:eigenvalues of A_n}
\alpha_l=\alpha^{(n)}_l=2\cos\frac{l\pi}{n+1}\,,
\qquad 1\le l\le n,
\end{equation}
see e.g., \cite{Bapat2010,Brouwer-Haemers2012,Obata2017-book}.
For each $n\ge1$ we associate a polynomial $\Phi_n(x)$ defined by
\begin{equation}\label{04eqn:new polynomials}
\Phi_n(x)
=((n+1)x^2-6x-4n)\Tilde{U}_n(x)
+2(x+2)\Tilde{U}_{n-1}(x)
+2(x+2).
\end{equation}
Let $\Tilde\alpha_n$ denote the minimal root of $\Phi_n(x)=0$.
(It will be shown that all roots of $\Phi_n(x)=0$ are real.)
With these notations we may state the main result.

\begin{theorem}\label{04thm:main formula for fan}
For $n\ge1$ we have
\begin{equation}\label{04eqn:main formula for a fan graph}
\mathrm{QEC}(K_1+P_n)
=-\Tilde\alpha_n-2,
\end{equation}
where $\Tilde\alpha_n$ is the minimal root of $\Phi_n(x)=0$.
If $n\ge2$ is even, then
$\Tilde\alpha_n=\min\mathrm{ev}(A_n)$ and 
\begin{equation}\label{04eqn:main formula for even n}
\mathrm{QEC}(K_1+P_n)
=-\min\mathrm{ev}(A_n)-2
=-4\sin^2\frac{\pi}{2(n+1)}\,.
\end{equation}
If $n\ge1$ is odd, 
then $\min\mathrm{ev}(A_{n+1})\le 
\Tilde{\alpha}_n<\min\mathrm{ev}(A_n)$ and 
\begin{equation}\label{04eqn:estimate for odd n}
-4\sin^2\frac{\pi}{2(n+1)}<\mathrm{QEC}(K_1+P_n)
\le -4\sin^2\frac{\pi}{2(n+2)}.
\end{equation}
\end{theorem}

For $n=1$ we have
\[
\Phi_1(x)=2(x-2)^2(x+1),
\qquad \Tilde{\alpha}_1=-1,
\qquad \min\mathrm{ev}(A_1)=0.
\]
For $n=2$ we have
\[
\Phi_2(x)=3(x-2)^2(x+1)^2,
\qquad \Tilde{\alpha}_2=-1,
\qquad \min\mathrm{ev}(A_2)=-1.
\]
Then, taking \eqref{04eqn:K_1+P_1 and P_2} into account,
the assertion of Theorem \ref{04thm:main formula for fan} 
is directly verified for $n=1,2$.
The rest of this section is devoted to the proof of
Theorem \ref{04thm:main formula for fan} for $n\ge3$.
Upon applying Theorem \ref{03thm:general formula},
we need to determine $\Lambda_1$ and $\Lambda_3$ mentioned therein,
while $\Lambda_0=\Lambda_2=\emptyset$ is obvious.

For $\Lambda_1$ we start with the three-term recurrence equation:
\begin{equation}\label{04eqn:main 3-term recurrence relation}
g_{k+2}-\alpha g_{k+1}+g_k=1,
\qquad 0\le k\le n-1,
\end{equation}
with boundary condition:
\begin{equation}\label{04eqn:main boundary condition}
g_0=g_{n+1}=0.
\end{equation}
Let $\xi$ and $\eta$ be the characteristic roots 
associated to \eqref{04eqn:main 3-term recurrence relation},
namely, $\xi$ and $\eta$ are complex numbers specified by
\begin{equation}\label{04eqn:characteristic roots}
\xi+\eta=\alpha,
\qquad
\xi\eta=1.
\end{equation}

\begin{lemma}\label{04lem:for Lambda1 (1)}
For $\alpha\not\in \mathrm{ev}(A_n)\cup\{\pm2\}$ 
let $\xi,\eta\in\mathbb{C}$ be defined by
\eqref{04eqn:characteristic roots}.
Then the boundary value problem 
\eqref{04eqn:main 3-term recurrence relation},
\eqref{04eqn:main boundary condition}
admits a unique solution given by
\begin{equation}\label{04eqn:explicit gk}
g_k=\frac{1}{2-\alpha}\bigg(
 1-\frac{\xi^k}{1+\xi^{n+1}}
  -\frac{\eta^k}{1+\eta^{n+1}}\bigg),
\qquad
0\le k\le n+1.
\end{equation}
\end{lemma}

\begin{proof}
Straightforward by general theory,
see Theorem \ref{00thm:lambda notin ev(A_n)} in Appendix.
\end{proof}

\begin{lemma}\label{04lem:Sigma}
For $\alpha\not\in \mathrm{ev}(A_n)\cup\{\pm2\}$ let 
$\xi,\eta\in\mathbb{C}$ be defined by
\eqref{04eqn:characteristic roots}.
Then we have
\begin{equation}\label{04eqn:sum of powers of xi and eta}
\sum_{k=1}^n\bigg(
\frac{\xi^k}{1+\xi^{n+1}}
+\frac{\eta^k}{1+\eta^{n+1}}\bigg)
=\frac{2}{2-\alpha}
\left(-1+\frac{\xi^n-\eta^n}{\xi^{n+1}-\eta^{n+1}}
+\frac{\xi-\eta}{\xi^{n+1}-\eta^{n+1}}\right).
\end{equation}
\end{lemma}

\begin{proof}
We first note from $\alpha\neq \pm2$ that
$\xi\neq \eta$ and $\xi,\eta\not\in\{\pm1, 0\}$.
Moreover, we see from $\alpha\not\in\mathrm{ev}(A_n)$ 
that $\xi^{n+1}\neq \eta^{n+1}$
and $\xi^{n+1}, \eta^{n+1}\not\in\{\pm1\}$.
With the help of $\xi\eta=1$ we obtain
\[
\frac{1}{1+\xi^{n+1}}
=\frac{1-\eta^{n+1}}{(1+\xi^{n+1})(1-\eta^{n+1})}
=\frac{1-\eta^{n+1}}{\xi^{n+1}-\eta^{n+1}}\,,
\]
and hence
\begin{equation}\label{04eqn:sum xi part of gk (1)} 
\sum_{k=1}^n \frac{\xi^k}{1+\xi^{n+1}}
=\frac{1-\eta^{n+1}}{\xi^{n+1}-\eta^{n+1}}\, 
  \frac{\xi(1-\xi^n)}{1-\xi}
=\frac{\xi-\eta^n-\xi^{n+1}+1}
{(\xi^{n+1}-\eta^{n+1})(1-\xi)}\,.
\end{equation}
Similarly,
\begin{equation}\label{04eqn:sum eta part of gk (1)}
\sum_{k=1}^n \frac{\eta^k}{1+\eta^{n+1}}
=\frac{\eta-\xi^n-\eta^{n+1}+1}
 {(\eta^{n+1}-\xi^{n+1})(1-\eta)}\,.
\end{equation}
Taking the sum of \eqref{04eqn:sum xi part of gk (1)}
and \eqref{04eqn:sum eta part of gk (1)}, we obtain
\begin{align*}
\sum_{k=1}^n\bigg(
\frac{\xi^k}{1+\xi^{n+1}}
&+\frac{\eta^k}{1+\eta^{n+1}}\bigg)
=\frac{\xi-\eta^n-\xi^{n+1}+1}
 {(\xi^{n+1}-\eta^{n+1})(1-\xi)}
 +\frac{\eta-\xi^n-\eta^{n+1}+1}
 {(\eta^{n+1}-\xi^{n+1})(1-\eta)} 
\\[3pt]
&=\frac{2\{-(\xi^{n+1}-\eta^{n+1})
           +(\xi^n-\eta^n)+(\xi-\eta)\}}
       {(\xi^{n+1}-\eta^{n+1})(1-\xi)(1-\eta)}
\\[3pt]
&=\frac{2}{2-\alpha}
  \frac{-(\xi^{n+1}-\eta^{n+1})
  +(\xi^n-\eta^n)+(\xi-\eta)}{\xi^{n+1}-\eta^{n+1}},
\end{align*}
from which \eqref{04eqn:sum of powers of xi and eta} follows.
\end{proof}

\begin{lemma}\label{04lem:Chebyshev}
For $\alpha\neq \pm2$ let $\xi,\eta\in\mathbb{C}$ be 
defined by \eqref{04eqn:characteristic roots}.
Then we have
\begin{equation}\label{04eqn:relation to U_n}
\frac{\xi^{n+1}-\eta^{n+1}}{\xi-\eta}
=\Tilde{U}_n(\alpha),
\qquad n\ge0.
\end{equation}
\end{lemma}

\begin{proof}
Note first that $\xi\neq\eta$ by $\alpha\neq\pm2$.
Since the left-hand side of \eqref{04eqn:relation to U_n}
is a symmetric polynomial in $\xi$ and $\eta$,
it is expressible in terms of $\xi+\eta=\alpha$ and $\xi\eta=1$,
and as a result,
becomes a polynomial in $\alpha$, which we denote by
$P_n(\alpha)$.
Obviously,
\begin{equation}\label{04eqn:P0 and P1}
P_0(\alpha)=1,
\qquad
P_1(\alpha)=\alpha.
\end{equation}
Moreover, for $n\ge1$ we have
\[
\alpha P_n(\alpha)
=(\xi+\eta)\,\frac{\xi^{n+1}-\eta^{n+1}}{\xi-\eta}
=\frac{\xi^{n+2}-\eta^{n+2}+\xi^n-\eta^n}{\xi-\eta}\,,
\]
where we used $\xi\eta=1$, and hence
\begin{equation}\label{04eqn:Pn}
\alpha P_n(\alpha)
=P_{n+1}(\alpha)+P_{n-1}(\alpha).
\end{equation}
On the other hand, $\Tilde{U}_n(\alpha)$ fulfills the
same recurrence relations as in
\eqref{04eqn:P0 and P1} and \eqref{04eqn:Pn}.
Consequently, $P_n(\alpha)=\Tilde{U}_n(\alpha)$ for all $n\ge0$.
\end{proof}

\begin{proposition}\label{04prop:Lambda1}
Let $n\ge3$.
Let $\Tilde{\Lambda}_1$ be the set of $\alpha\in\mathrm{R}$
satisfying
\begin{equation}\label{04eqn:main eqn for Lambda1}
(\alpha+2)\langle\bm{1},(A_n-\alpha I)^{-1}\bm{1}\rangle-1=0
\end{equation}
and put $\Lambda_1=\Tilde{\Lambda}_1
\backslash(\mathrm{ev}(A_n)\cup\{0,-1,-2\})$.
Then we have
\begin{equation}\label{04eqn:Lambda1-2}
\Lambda_1\backslash\{2\}
=\{x\in\mathrm{R}\,;\, 
\Phi_n(x)=0\}
\backslash(\mathrm{ev}(A_n)\cup\{0,-1,\pm2\}).
\end{equation}
\end{proposition}

\begin{proof}
It is sufficient to prove that
equation \eqref{04eqn:main eqn for Lambda1} is equivalent to
$\Phi_n(\alpha)=0$
whenever $\alpha\in\mathbb{R}\backslash(\mathrm{ev}(A_n)\cup\{\pm2\})$.
Assume that 
$\alpha\in\mathbb{R}\backslash(\mathrm{ev}(A_n)\cup\{\pm2\})$.
We set $g=(A_n-\alpha I)^{-1}\bm{1}$ and $g=[g_1, \dots, g_n]^T$.
Obviously, we have
\begin{equation}\label{04eqn:sum gk}
\langle\bm{1},(A_n-\alpha I)^{-1}\bm{1}\rangle
=\langle\bm{1},g\rangle
=\sum_{k=1}^n g_k\,.
\end{equation}
On the other hand,
$g=(A_n-\alpha I)^{-1}\bm{1}$ is equivalent to
$(A_n-\alpha I)g=\bm{1}$ and so is to the three-term recurrence 
relation \eqref{04eqn:main 3-term recurrence relation}
with boundary condition
\eqref{04eqn:main boundary condition}.
The unique solution $\{g_k\}$ is described in Lemma
\ref{04lem:for Lambda1 (1)}.
Then, it follows from Lemma \ref{04lem:Sigma} that
\begin{align}
\sum_{k=1}^n g_k
&=\frac{1}{2-\alpha}
 \Bigg\{n-
 \sum_{k=1}^n\Bigg(\frac{\xi^k}{1+\xi^{n+1}}
             +\frac{\eta^k}{1+\eta^{n+1}}\Bigg)
 \Bigg\} 
\nonumber \\
&=\frac{1}{2-\alpha}
 \Bigg\{n-
 \frac{2}{2-\alpha}
 \Bigg(-1+\frac{\xi^n-\eta^n}{\xi^{n+1}-\eta^{n+1}}
 +\frac{\xi-\eta}{\xi^{n+1}-\eta^{n+1}}\Bigg)
\Bigg\} 
\nonumber \\
&=\frac{1}{(2-\alpha)^2}
 \Bigg\{n(2-\alpha)+2
 -2\Bigg(\frac{\xi^n-\eta^n}{\xi^{n+1}-\eta^{n+1}}
 +\frac{\xi-\eta}{\xi^{n+1}-\eta^{n+1}}\Bigg)
\Bigg\}.
\label{04eqn:in proof Prop4.1(1)}
\end{align}
Moreover, applying
\[
\frac{\xi^n-\eta^n}{\xi^{n+1}-\eta^{n+1}}
 +\frac{\xi-\eta}{\xi^{n+1}-\eta^{n+1}}
=\frac{\Tilde{U}_{n-1}(\alpha)}{\Tilde{U}_n(\alpha)}
 +\frac{1}{\Tilde{U}_n(\alpha)}\,,
\]
which follows from Lemma \ref{04lem:Chebyshev},
we combine 
\eqref{04eqn:sum gk} and \eqref{04eqn:in proof Prop4.1(1)}
to obtain
\[
\langle\bm{1},(A_n-\alpha)^{-1}\bm{1}\rangle
=\frac{1}{(2-\alpha)^2}
 \Bigg\{n(2-\alpha)+2
 -2\Bigg(\frac{\Tilde{U}_{n-1}(\alpha)}{\Tilde{U}_n(\alpha)}
 +\frac{1}{\Tilde{U}_n(\alpha)}\Bigg)
\Bigg\}.
\]
Then, by simple algebra we see that
equation \eqref{04eqn:main eqn for Lambda1} is
equivalent to 
\[
((n+1)\alpha^2-6\alpha-4n)\Tilde{U}_n(\alpha)
+2(\alpha+2)\Tilde{U}_{n-1}(\alpha)
+2(\alpha+2)
=0,
\]
that is, $\Phi_n(\alpha)=0$ as desired,
see \eqref{04eqn:new polynomials}.
\end{proof}

For $\Lambda_3$ we start with the following lemma.

\begin{lemma}\label{04lem:eigenvectors of Pn}
For $n\ge1$
the eigenvalues of the adjacency matrix $A_n$ of $P_n$ are
given by
\[
\alpha^{(n)}_l=2\cos\frac{l\pi}{n+1}\,,
\qquad 1\le l\le n.
\]
Moreover, the eigenspace associated to $\alpha^{(n)}_l$
is orthogonal to $\bm{1}$ if and only if $l$ is even.
\end{lemma}

\begin{proof}
The first half of the assertion is already mentioned (well known),
see \eqref{04eqn:ev(A_n)} and \eqref{04eqn:eigenvalues of A_n}.
For an eigenvalue $\alpha=\alpha^{(n)}_l$ the equation
$Ag=\alpha g$ is transferred into
the three-term recurrence equation:
\[
g_{k+2}-\alpha g_{k+1}+g_k=0,
\qquad 0\le k\le n-1,
\]
with boundary condition $g_0=g_{n+1}=0$.
The solutions are described explicitly
by general theory (Theorem \ref{00thm:lambda in ev(A_n)} in Appendix).
As a result, any eigenvector associated to
an eigenvalue $\alpha=\alpha^{(n)}_l$ is
a constant multiple of $g=[g_1,\dots,g_n]^T$ given by
\[
g_k=\sin\frac{kl\pi}{n+1}\,,
\qquad 1\le k\le n\,.
\]
Moreover, with the help of elementary formula
of trigonometric series we obtain
\[
\langle\bm{1},g\rangle
=\sum_{k=1}^n \sin\frac{kl\pi}{n+1}
=\frac{\sin\dfrac{l\pi}{2}\sin\dfrac{nl\pi}{2(n+1)}}
 {\sin\dfrac{l\pi}{2(n+1)}}\,.
\]
Since $nl/2(n+1)$ is not an integer for $1\le l\le n$, 
we see that
$\langle\bm{1},g\rangle=0$ if and only if $l$ is even.
\end{proof}

\begin{remark}
\normalfont
By reflection symmetry of $P_n$ every eigenvector of $A_n$
is palindromic or skew palindromic.
It is noteworthy that an eigenvector of $A_n$
is orthogonal to $\bm{1}$ if and only if
it is skew palindromic,
see \cite{Doob-Haemers2002}.
\end{remark}

\begin{proposition}\label{04prop:Lambda3}
Let $n\ge3$.
Let $\Lambda_3$ be the set of $\alpha\in\mathrm{ev}(A_n)
\backslash\{0,-1,-2\}$ which admits an eigenvector $g$
such that $\langle\bm{1},g\rangle=0$.
Then 
\[
\Lambda_3=\{\alpha^{(n)}_l\,;\,
1\le l\le n,\,\,\, \text{$l$ is even}\}
\backslash\{0,-1\}.
\]
\end{proposition}

\begin{proof}
Immediate from Lemma \ref{04lem:eigenvectors of Pn}.
\end{proof}

We now apply the formula in Theorem \ref{03thm:general formula}
to $K_1+P_n$ with $n\ge3$.
Since $\Lambda_0=\Lambda_2=\emptyset$ for
$K_1+P_n$ with $n\ge3$,
we have
\begin{equation}\label{04eqn:in proof Theorem 4.1}
\mathrm{QEC}(K_1+P_n)
=-\Tilde{\alpha}-2,
\end{equation}
where
\begin{align}
\Tilde{\alpha}
&=\min (\Lambda_1\cup \Lambda_3)\cap(-\infty,-1) 
\nonumber \\
&=\min \{\min\Lambda_1\cap(-\infty,-1),\,\,
   \min\Lambda_3\cap(-\infty,-1)\,\}.
\label{04eqn:Tilde_alpha after Proposition 4.7}
\end{align}
By Propositions \ref{04prop:Lambda1} we have
\begin{equation}\label{04eqn:min Lambda1 (1)}
\min\Lambda_1\cap(-\infty,-1)
=\min \big(\{x\in\mathbb{R}\,;\, \Phi_n(x)=0\}
 \backslash(\mathrm{ev}(A_n)\cup\{-2\})\big)
\end{equation}
and by Proposition \ref{04prop:Lambda3} we have
\begin{equation}\label{04eqn:min Lambda3 (2)}
\min\Lambda_3\cap(-\infty,-1)=\alpha^{(n)}_l,
\end{equation}
where $l$ is the maximal even integer such that
$1\le l\le n$ and $\alpha^{(n)}_l\neq0,-1$.

For comparison of
\eqref{04eqn:min Lambda1 (1)} 
and \eqref{04eqn:min Lambda3 (2)}
we need to examine the zeros of the polynomials $\Phi_n(x)$.

\begin{lemma}\label{04lem:Phi(2)}
For $n\ge1$ the polynomial $\Phi_n(x)$ is of degree $n+2$
and we have
\begin{equation}\label{04eqn:properties of Phi_n}
\Phi_n(2)=\Phi_n^\prime(2)=0,
\qquad
\Phi_n(-2)=16(n+1)(-1)^n.
\end{equation}
\end{lemma}

\begin{proof}
The assertions are verified easily by direct calculation 
using explicit form \eqref{04eqn:new polynomials} of $\Phi_n(x)$
and its derivative:
\begin{align}
\Phi^\prime_n(x)
&=2((n+1)x-3)U_n\bigg(\frac{x}{2}\bigg)
 +\frac12((n+1)x^2-6x-4n)U_n^\prime\bigg(\frac{x}{2}\bigg) 
\nonumber \\
&\qquad +2U_{n-1}\bigg(\frac{x}{2}\bigg)
 +(x+2)U_{n-1}^\prime\bigg(\frac{x}{2}\bigg)
 +2,
\label{04eqn:Phi_n^prime}
\end{align}
with the help of the following formulas for Chebyshev polynomials:
for $1\le l \le n$ we have
\begin{gather}
U_n\bigg(\cos\frac{l\pi}{n+1}\bigg)=0,
\quad
U_n^\prime\bigg(\cos\frac{l\pi}{n+1}\bigg)
=\frac{(-1)^{l+1}(n+1)}{\sin^2\dfrac{l\pi}{n+1}}\,,
\label{04eqn:Un and Un^prime} \\
U_{n-1}\bigg(\cos\frac{l\pi}{n+1}\bigg)=(-1)^{l+1},
\quad
U_{n-1}^\prime\bigg(\cos\frac{l\pi}{n+1}\bigg)
=\frac{(-1)^{l+1}(n+1)\cos\dfrac{l\pi}{n+1}}
 {\sin^2\dfrac{l\pi}{n+1}}\,,
\label{04eqn:Un and Un-1^prime}
\end{gather}
In fact, the above formulas are verified by definition.
\end{proof}

We next examine the values of $\Phi_n(x)$ at 
\[
\alpha^{(n)}_l=2\cos\frac{l\pi}{n+1}\,,
\qquad 1\le l\le n.
\]

\begin{lemma}
For $1\le l\le n$ we have
\begin{align}
\Phi_n(\alpha^{(n)}_l)
&=8((-1)^{l+1}+1)\cos^2\frac{l\pi}{n+1}\,,
\label{04eqn:Phi_n(alpha)} \\
\Phi_n^\prime(\alpha^{(n)}_l)
&=\frac{2(-1)^l(n+1)(n+2)\bigg(1-\cos\dfrac{l\pi}{n+1}\bigg)}
       {\sin^2\dfrac{l\pi}{n+1}}
  \bigg(\cos\frac{l\pi}{n+1}+\frac{n}{n+2}\bigg) 
\nonumber \\
&\quad  + 2((-1)^{l+1}+1).
\label{04eqn:Phi_n^prime(alpha)}
\end{align}
\end{lemma}

\begin{proof}
By direct calculation with the help of
\eqref{04eqn:Phi_n^prime}--\eqref{04eqn:Un and Un-1^prime}.
\end{proof}

\begin{lemma}\label{04lem:Phi_n(alpha)}
Let $n\ge3$ and $1\le l\le n$.
If $l$ is odd, we have
\begin{equation}\label{04eqn:Phi_n(alpha) with odd l}
\Phi_n(\alpha^{(n)}_l)=16\cos^2\frac{l\pi}{n+1}>0.
\end{equation}
If $l$ is even, we have
\begin{equation}\label{04eqn:Phi_n(alpha) with even l}
\Phi_n(\alpha^{(n)}_l)=0,
\qquad
\Phi_n^\prime(\alpha^{(n)}_l)
=K^{(n)}_l \bigg(\cos\frac{l\pi}{n+1}+\frac{n}{n+2}\bigg)
\neq0,
\end{equation}
where
\[
K^{(n)}_l 
=\frac{2(n+1)(n+2)\bigg(1-\cos\dfrac{l\pi}{n+1}\bigg)}
       {\sin^2\dfrac{l\pi}{n+1}}
>0.
\]
\end{lemma}

\begin{proof}
Equalities in \eqref{04eqn:Phi_n(alpha) with odd l} and
\eqref{04eqn:Phi_n(alpha) with even l} follow
immediately from Lemma \ref{04lem:Phi_n(alpha)}.
That $\Phi_n^\prime(\alpha^{(n)}_l)\neq0$ for an even $l$
follows from the elementary fact
(sometimes referred to as Niven's theorem)
that $\{\cos r\pi\,;\, r\in\mathbb{Q}\}\cap\mathbb{Q}
=\{0,\pm1/2,\pm1\}$, 
where $\mathbb{Q}$ is the set of rational numbers.
In fact, we have
\[
\cos\frac{l\pi}{n+1}+\frac{n}{n+2}\neq0,
\]
whenever $n\ge3$ and $1\le l\le n$.
\end{proof}

\begin{proposition}\label{04prop:Phi_n(x)}
Let $n\ge3$.
All roots of $\Phi_n(x)=0$ are real,
and they are simple except $x=2$ (this is a double root).
Let $\Tilde\alpha_n$ be the minimal root of $\Phi_n(x)=0$.
If $n$ is even, we have
\begin{equation}\label{04eqn:even case}
\Tilde\alpha_n
=\alpha^{(n)}_n=2\cos\frac{n\pi}{n+1}=-2\cos\frac{\pi}{n+1}\,.
\end{equation}
If $n$ is odd, we have
\begin{equation}\label{04eqn:odd case}
-2<\Tilde\alpha_n<\alpha^{(n)}_n\,.
\end{equation}
\end{proposition}

\begin{proof}
First consider the case where $n\ge3$ is even.
It follows from Lemma \ref{04lem:Phi_n(alpha)} that
$\alpha^{(n)}_2,\alpha^{(n)}_4,\dots,\alpha^{(n)}_n$ are
simple roots of $\Phi_n(x)=0$.
We note that
\begin{equation}\label{04eqn:Phi_n^prime 2 and n}
\Phi_n^\prime(\alpha^{(n)}_2)>0,
\qquad
\Phi_n^\prime(\alpha^{(n)}_n)<0.
\end{equation}
In fact, the first inequality is obvious
by \eqref{04eqn:Phi_n(alpha) with even l}.
For the second, using an elementary inequality:
\[
\cos\theta>\frac{\pi-\theta}{\pi+\theta},
\qquad 0<\theta<\frac{\pi}{3}\,,
\]
we have
\[
\cos\frac{n\pi}{n+1}=-\cos\frac{\pi}{n+1}
<-\frac{\pi-\pi/(n+1)}{\pi+\pi/(n+1)}
=-\frac{n}{n+2}.
\]
Then the second inequality in \eqref{04eqn:Phi_n^prime 2 and n}
follows from \eqref{04eqn:Phi_n(alpha) with even l}.
Hence, we see from \eqref{04eqn:Phi_n(alpha) with even l} and
\eqref{04eqn:Phi_n^prime 2 and n}
that there exists an even integer $m$ with $2\le m<n$, such that
\begin{align*}
&\Phi_n^\prime(\alpha^{(n)}_2)>0,
\,\,\, \Phi_n^\prime(\alpha^{(n)}_4)>0,
\,\,\,\dots,\,\,\, \Phi_n^\prime(\alpha^{(n)}_m)>0, \\
&\hspace{100pt}
\Phi_n^\prime(\alpha^{(n)}_{m+2})<0,
\,\,\,\dots,\,\,\, \Phi_n^\prime(\alpha^{(n)}_n)<0.
\end{align*}
Since $\Phi_n(\alpha^{(n)}_l)>0$ for any odd $l$
and $\Phi_n(\alpha^{(n)}_l)=0$ for any even $l$
by Lemma \ref{04lem:Phi_n(alpha)},
it follows from the intermediate value theorem
that each of the intervals:
\[
(\alpha^{(n)}_3,\alpha^{(n)}_2),
\,\,\,(\alpha^{(n)}_5,\alpha^{(n)}_4),
\,\,\,\dots, 
\,\,\,(\alpha^{(n)}_{m+1},\alpha^{(n)}_m),
\,\,\,(\alpha^{(n)}_{m+2},\alpha^{(n)}_{m+1}),
\,\,\,\dots, 
\,\,\,(\alpha^{(n)}_n,\alpha^{(n)}_{n-1})
\]
contains at least one root of $\Phi_n(x)=0$.
The number of those roots is at least $n/2$.
On the other hand, $\alpha^{(n)}_l$ with an even $l$ is
a simple root by Lemma \ref{04lem:Phi_n(alpha)}
and the number of such roots is $n/2$.
We know from Lemma \ref{04lem:Phi(2)} 
that $x=2$ is a root of $\Phi_n(x)=0$ with
multiplicity at least $2$. 
Thus, the total number of roots is already at least $n+2$.
Since the degree of $\Phi_n(x)$ is $n+2$,
there are no other roots and
we conclude that all roots of $\Phi_n(x)=0$ are real,
they are simple except $x=2$ which is a double root.
Consequently, the minimal root is $x=\alpha^{(n)}_n$
and \eqref{04eqn:even case} holds.

We next suppose that $n$ is odd.
The argument is similar as in the case of an even $n$.
We only need to note that
$\Phi_n(\alpha_n^{(n)})>0$,
$\Phi_n(-2)<0$, 
and $\Phi_n^\prime(\alpha^{(n)}_n)>0$ may happen
(in fact, $\Phi_n^\prime(\alpha^{(n)}_{n-1})>0$ for $n\le 10$
and $\Phi_n^\prime(\alpha^{(n)}_{n-1})>0$ for $n\ge11$).
In particular, the interval $(-2, \alpha^{(n)}_n)$ contains
at least one root of $\Phi_n(x)=0$.
After counting the number of roots 
in a similar manner as in the case of an even $n$,
we conclude that
the interval $(-2, \alpha^{(n)}_n)$ contains just one simple root,
which is the minimal root of $\Phi_n(x)=0$.
Thus \eqref{04eqn:odd case} follows.
\end{proof}

\begin{proof}[Proof of Theorem \ref{04thm:main formula for fan}]
We examine the right-hand side of \eqref{04eqn:Tilde_alpha after Proposition 4.7}.
As is mentioned there 
we see from Proposition \ref{04prop:Lambda1} that
$\min\Lambda_1$ is the minimal root of $\Phi_n(x)=0$ such that
$x\notin\mathrm{ev}(A_n)\cup\{0,-1,\pm2\}$.
By Proposition \ref{04prop:Phi_n(x)} we have
\begin{equation}\label{04eqn:min Lambda1}
\begin{array}{ll}
\min\Lambda_1\in (\alpha^{(n)}_n, \alpha^{(n)}_{n-1}),
&\text{if $n$ is even},\\
\min\Lambda_1\in (-2, \alpha^{(n)}_n),
&\text{if $n$ is odd}.
\end{array}
\end{equation}
On the other hand, by Proposition \ref{04prop:Lambda3} 
we see that
\[
\min \Lambda_3=\alpha^{(n)}_l=2\cos\frac{l\pi}{n+1}\,,
\]
where $l$ is the maximal even integer such that
$1\le l\le n$ and $\alpha^{(n)}_l\neq-1$.
Then
\begin{equation}\label{04eqn:min Lambda3}
\begin{array}{ll}
\min \Lambda_3=\alpha^{(n)}_n,
&\text{if $n$ is even},\\
\min \Lambda_3=\alpha^{(n)}_l \,\,\,\text{with some $l<n$},
&\text{if $n$ is odd}.
\end{array}
\end{equation}
Comparing \eqref{04eqn:min Lambda1} and \eqref{04eqn:min Lambda3},
we obtain
\[
\min\{\min\Lambda_1,\min\Lambda_3\}
=\begin{cases}
\alpha^{(n)}_n, & \text{if $n$ is even}, \\
\min \Lambda_1, & \text{if $n$ is odd}.
\end{cases}
\]
Recall that if $n$ is even, 
the minimal root of $\Phi_n(x)=0$ coincides with $\alpha^{(n)}_n$,
that is, $\Tilde{\alpha}_n=\alpha_n^{(n)}$.
Therefore, regardless of whether $n$ is even or odd,
$\min\{\min\Lambda_1,\min\Lambda_3\}$ coincides with
the minimal root of $\Phi_n(x)=0$, 
that is,
\[
\min\{\min\Lambda_1,\,\min\Lambda_3\}=\Tilde\alpha_n.
\]
Consequently,
\[
\mathrm{QEC}(K_1+P_n)
=-\Tilde\alpha_n-2,
\]
which proves \eqref{04eqn:main formula for a fan graph}.

Suppose that $n\ge2$ is even.
In view of \eqref{04eqn:even case}
in Proposition \ref{04prop:Phi_n(x)} we have
$\Tilde{\alpha}_n=\alpha^{(n)}_n$ and 
\begin{equation}\label{04eqn:in proof Thm4.1(1)}
\mathrm{QEC}(K_1+P_n)
=-\Tilde{\alpha}_n-2
=2\cos\frac{\pi}{n+1}-2
=-4\sin^2\frac{\pi}{2(n+1)}\,,
\end{equation}
which proves \eqref{04eqn:main formula for even n}.
Suppose that $n\ge3$ is odd.
In view of \eqref{04eqn:odd case} in
Proposition \ref{04prop:Phi_n(x)} we have
\[
\mathrm{QEC}(K_1+P_n)
=-\Tilde{\alpha}_n-2
>-\alpha^{(n)}_n-2
=-4\sin^2\frac{\pi}{2(n+1)}.
\]
On the other hand,
we see from \eqref{04eqn:K_1+P_n in K_1+P_n+1} and
\eqref{04eqn:in proof Thm4.1(1)} that
\[
\mathrm{QEC}(K_1+P_n)\le \mathrm{QEC}(K_1+P_{n+1})
=-\Tilde{\alpha}_{n+1}-2
=-4\sin^2\frac{\pi}{2(n+2)}\,,
\]
where we note that $n+1$ is even.
Thus \eqref{04eqn:estimate for odd n} follows.
\end{proof}

\begin{remark}
\normalfont
During the above argument we have seen that
the minimal zeros $\Tilde{\alpha}_n$ of $\Phi_n(x)$ form
a decreasing sequence satisfying
\[
\Tilde{\alpha}_1= \Tilde{\alpha}_2
>\Tilde{\alpha}_3
\ge \Tilde{\alpha}_4
>\Tilde{\alpha}_5
\ge \Tilde{\alpha}_6
>\dotsb \rightarrow -2.
\]
It is plausible that the above inequalities are all strict.
\end{remark} 

\begin{remark}\label{04rem:K1+Pn is QE}
\normalfont
It follows from Theorem \ref{04thm:main formula for fan}
that $\mathrm{QEC}(K_1+P_n)<0$, 
namely the fan graph $K_1+P_n$ admits
a quadratic embedding in a Euclidean space.
While, a quadratic embedding is described explicitly.
Let $e_1,e_2,\dots,e_n$ be the orthonormal
basis of the Euclidean space $\mathbb{R}^n$
and define $x_0,x_1,\dots, x_n\in \mathbb{R}^n$ by
\begin{align*}
x_0&=0, \\
x_k&=\sqrt{\frac{k-1}{2k}}\, e_{k-1}
    + \sqrt{\frac{k+1}{2k}}\, e_k\,,
    \qquad 1\le k\le n.
\end{align*}
Consider $K_1$ as the singleton graph with vertex set $V_1=\{0\}$
and $P_n$ as the path with vertex set $V_2=\{1,2,\dots,n\}$
in a natural manner.
Then the map $\varphi: V_1\cup V_2 \rightarrow \mathbb{R}^n$
defined by $\varphi(k)=x_k$ for $0\le k \le n$
is a quadratic embedding of $K_1+P_n$
in the Euclidean space $\mathbb{R}^n$.
In fact, it is easy to see that
$\|\varphi(j)-\varphi(k)\|^2=d(j,k)$ for $0\le j,k\le n$.
\end{remark}

\section*{Appendix A: Three-Term Recurrence Equations}
\setcounter{section}{1}
\setcounter{equation}{0}
\setcounter{theorem}{0}
\renewcommand{\thesection}{\Alph{section}}
\renewcommand{\theequation}{\Alph{section}.\arabic{equation}}

Given an integer $n\ge1$ and
real constants $\lambda, \mu\in\mathbb{R}$ we
consider the three-term recurrence equation:
\begin{equation}\label{00eqn: main}
f_{k+2}-\lambda f_{k+1}+f_k =\mu,
\qquad 0\le k\le n-1, 
\end{equation}
with zero-boundary conditions:
\begin{equation}\label{00eqn: zero-boundary condition}
f_0=f_{n+1}=0.
\end{equation}
The above boundary value problem belongs to elementary mathematics
and the solution is well known in various forms.
We here portray the solution $\{f_k\,;\, 0\le k\le n+1\}$ 
in a way that aligns with our argument in Section \ref{04sec:K1+Pn}.

With \eqref{00eqn: main} we associate a characteristic equation
$\xi^2-\lambda\xi+1=0$, of which roots are denoted by $\xi$ and $\eta$.
Without loss of generality we may set
\[
\xi=\frac{\lambda+\sqrt{\lambda^2-4}}{2},
\qquad
\eta=\frac{\lambda-\sqrt{\lambda^2-4}}{2}.
\]
In fact, $\xi,\eta\in\mathbb{C}$ are characterized by
\[
\xi+\eta=\lambda,
\qquad
\xi\eta=1.
\]
Note also that $\xi=\eta=1$ if $\lambda=2$,
$\xi=\eta=-1$ if $\lambda=-2$,
and $\xi\neq\eta$ otherwise.

\begin{proposition}[Homogeneous case]\label{00prop:homogeneous case}
A general solution to the three-term recurrence equation
\eqref{00eqn: main} with $\mu=0$ is given by
\begin{align*}
f_k=
\begin{cases}
K_1+K_2k, & \text{if $\lambda=2$}, \\
(K_1+K_2k)(-1)^k, & \text{if $\lambda=-2$}, \\
K_1\xi^k+K_2\eta^k, & \text{otherwise},
\end{cases}
\end{align*}
where $K_1$ and $K_2$ are arbitrary constants.
\end{proposition}

\begin{theorem}[General case with $\lambda=\pm2$]
\label{00thm:lambda=pm2}
{\upshape (1)} For $\lambda=2$, 
the boundary value problem \eqref{00eqn: main},
\eqref{00eqn: zero-boundary condition} has a unique solution given by
\[
f_k=-\frac{\mu}{2}\,(n+1)k+\frac{\mu}{2}\,k^2.
\]
{\upshape (2)} For $\lambda=-2$, 
the boundary value problem \eqref{00eqn: main},
\eqref{00eqn: zero-boundary condition} has a unique solution given by
\[
f_k=\frac{\mu}{4}\,((1-(-1)^k)
 +\frac{\mu}{4(n+1)}\,(1+(-1)^n)(-1)^kk.
\]
\end{theorem}

\begin{proposition}
Let $\lambda\neq\pm2$.
A general solution to \eqref{00eqn: main} is given by
\begin{equation}\label{00eqn:general solution to lambda neq pm2}
f_k=K_1\xi^k+K_2\eta^k+\frac{\mu}{2-\lambda}\,,
\end{equation}
where $K_1$ and $K_2$ are arbitrary constants.
\end{proposition}

Thus, for the solution to the boundary value problem \eqref{00eqn: main},
\eqref{00eqn: zero-boundary condition} it is sufficient 
to determine the constants $K_1$ and $K_2$
in \eqref{00eqn:general solution to lambda neq pm2}
so as to fulfill the boundary condition \eqref{00eqn: zero-boundary condition}.
Namely,
\begin{align*}
f_0&=K_1+K_2+\frac{\mu}{2-\lambda}=0, 
\\
f_{n+1}&=K_1\xi^{n+1}+K_2\eta^{n+1}+\frac{\mu}{2-\lambda}=0,
\end{align*}
or in a matrix form:
\begin{equation}\label{00eqn:for K_1 and K_2 (3)}
\begin{bmatrix}
1 & 1 \\
\xi^{n+1} & \eta^{n+1}
\end{bmatrix}
\begin{bmatrix}
K_1 \\ K_2 
\end{bmatrix}
=\frac{\mu}{\lambda-2}\,\1.
\end{equation}
For convenience we set
\begin{equation}\label{00eqn:Delta}
\Delta
=\begin{vmatrix}
1 & 1 \\
\,\xi^{n+1} & \eta^{n+1}\,
\end{vmatrix}
=\eta^{n+1}-\xi^{n+1}.
\end{equation}

On the other hand, the boundary value problem 
\eqref{00eqn: main}, \eqref{00eqn: zero-boundary condition}
is equivalent to
\begin{equation}\label{00eqn:main problem by A}
(A_n-\lambda I)f=\mu\1,
\end{equation}
where 
\[
A_n=\begin{bmatrix}
0 & 1 &  \\
1 & 0 & 1 & \\
  & 1 & 0 & 1 & \\
  &         & \ddots  & \ddots & \ddots \\
  &         &         & 1       & 0 & 1 \\
  &         &         &        & 1 & 0
\end{bmatrix},
\qquad
f=\begin{bmatrix} f_1 \\ f_2 \\ \vdots \\ f_n \end{bmatrix},
\quad
\1=\begin{bmatrix} 1 \\ 1 \\ \vdots \\ 1 \end{bmatrix}.
\]
In fact, a real sequence $\{f_k\,;\, 0\le k\le n+1\}$
with $f_0=f_{n+1}=0$ is a solution to the boundary value problem 
\eqref{00eqn: main}, \eqref{00eqn: zero-boundary condition}
if and only if $f=[f_1, \dots, f_n]^T$ fulfills
\eqref{00eqn:main problem by A}.
It is noted that $A_n$ is nothing else but the adjacency matrix of $P_n$.
Recall that the eigenvalues of $A_n$ are given by
\[
\mathrm{ev}(A_n)
=\left\{
2\cos\frac{l\pi}{n+1}\,;\, 1\le l\le n\right\}.
\]
By direct calculation we obtain the following assertion.

\begin{lemma}\label{00lem:Delta=0}
For $\lambda\in\mathbb{R}$ let $\xi,\eta$ be
the characteristic roots of $\xi^2-\lambda\xi+1=0$.
Then $\Delta=\eta^{n+1}-\xi^{n+1}=0$ 
if and only if $\lambda\in\mathrm{ev}(A_n)\cup\{\pm2\}$.
\end{lemma}

Thus, upon specifying the constants $K_1$ and $K_2$ in
\eqref{00eqn:for K_1 and K_2 (3)}
we need to separate our argument
according as $\lambda\in\mathrm{ev}(A_n)$ or not.

\begin{theorem}[General case with 
$\lambda\not\in\mathrm{ev}(A_n)\cup\{\pm2\}$]
\label{00thm:lambda notin ev(A_n)}
If $\lambda\not\in\mathrm{ev}(A_n)\cup\{\pm2\}$,
then the boundary value problem \eqref{00eqn: main},
\eqref{00eqn: zero-boundary condition} has a unique solution given by
\[
f_k=\frac{\mu}{2-\lambda}
\left(1-\frac{\xi^k}{1+\xi^{n+1}}
 -\frac{\eta^k}{1+\eta^{n+1}}\right).
\]
\end{theorem}

\begin{theorem}[General case with $\lambda\in\mathrm{ev}(A_n)$]
\label{00thm:lambda in ev(A_n)}
Let $\lambda=2\cos l\pi/(n+1)$ with $1\le l\le n$.

{\upshape (1)} If $\mu=0$, 
then any solution to the boundary value problem \eqref{00eqn: main},
\eqref{00eqn: zero-boundary condition} is given by
\[
f_k=K\sin \frac{kl\pi}{n+1}\,,
\]
where $K\in\mathbb{R}$ is an arbitrary constant.

{\upshape (2)} If $\mu\neq 0$ and $l$ is even, 
then any solution to the boundary value problem \eqref{00eqn: main},
\eqref{00eqn: zero-boundary condition} is given by
\[
f_k=\frac{\mu}{2-\lambda}\left(1-\cos\frac{kl\pi}{n+1}\right)
 +K\sin\frac{kl\pi}{n+1}\,,
\]
where $K\in\mathbb{R}$ is an arbitrary constant.

{\upshape (3)} If $\mu\neq0$ and $l$ is odd,
then the boundary value problem \eqref{00eqn: main},
\eqref{00eqn: zero-boundary condition} has no solution.
\end{theorem}

\section*{Appendix B: Partial Chebyshev Polynomials}
\setcounter{section}{2}
\setcounter{equation}{0}
\setcounter{theorem}{0}
\renewcommand{\thesection}{\Alph{section}}
\renewcommand{\theequation}{\Alph{section}.\arabic{equation}}

In Section \ref{04sec:K1+Pn} we introduced polynomials
$\Phi_n(x)$, $n\ge1$, defined by
\begin{equation}\label{0beqn:def of Phi_n}
\Phi_n(x)
=\big((n+1)x^2-6x-4n\big)\Tilde{U}_n(x)
 +2(x+2)\Tilde{U}_{n-1}(x)+2(x+2),
\end{equation}
where $\Tilde{U}_n(x)=U_n(x/2)=x^n+\dotsb$ 
is the ``compressed'' Chebyshev polynomial of the second kind.
Note that $\Phi_n(x)$ is a polynomial of degree $n+2$ with
integer coefficients.
We will give an interesting factorization of $\Phi_n(x)$
in terms of Chebyshev polynomials.

For each $n\ge0$ we associate new polynomials $U^{\rm{e}}_n(x)$ and
$U^{\rm{o}}_n(x)$ defined by
\begin{align}
U^{\rm{e}}_n(\cos\theta)
&=\begin{cases}
\displaystyle\frac{\sin(n+1)\theta/2}{\sin\theta/2}\,,
&\text{if $n$ is even,}\\[12pt]
\displaystyle\frac{\sin (n+1)\theta/2}{\sin\theta}\,,
&\text{if $n$ is odd,}
\end{cases}
\label{0beqn:def of Ue} \\[6pt]
U^{\rm{o}}_n(\cos \theta)
&=\begin{cases}
\displaystyle\frac{\cos(n+1)\theta/2}{\cos\theta/2}\,,
&\text{if $n$ is even,}\\[12pt]
\displaystyle 2\cos\frac{(n+1)\theta}{2}\,,
&\text{if $n$ is odd.}
\end{cases}
\label{0beqn:def of Uo}
\end{align}
In fact, using elementary trigonometric formulas one
may show that the right-hand sides of
\eqref{0beqn:def of Ue} and \eqref{0beqn:def of Uo} are
polynomials of $\cos\theta$.
Moreover, it is easy to see that
\begin{equation}\label{0b:U=UeUo}
U_n(x)=U^{\rm{e}}_n(x) U^{\rm{o}}_n(x),
\qquad n\ge0.
\end{equation}
The polynomials $U^{\rm{e}}_n(x)$ and $U^{\rm{o}}_n(x)$ seem to be
new in literature and
are informally called ``partial Chebyshev polynomials.''

\begin{theorem}\label{0bthm:Ue and Uo by U}
For $n\ge0$ we have
\begin{align}
U^{\rm{e}}_{2n}(x)&=U_{n}(x)+U_{n-1}(x),
&U^{\rm{o}}_{2n}(x)&=U_n(x)-U_{n-1}(x),
\label{for:ueuoeven}\\
U^{\rm{e}}_{2n+1}(x)&=U_n(x),
&U^{\rm{o}}_{2n+1}(x)&=U_{n+1}(x)-U_{n-1}(x),
\label{for:ueuoodd}
\end{align}
under convention that $U_{-1}(x)=0$.
\end{theorem}

\begin{proof}
By using additive formula for trigonometric functions we have
\begin{align}
U^{\rm{e}}_{2n}(\cos\theta)
&=\frac{\sin (2n+1)\theta/2}{\sin\theta/2}
=\frac{\sin (2n-1)\theta/2 \cos\theta+\cos (2n-1)\theta/2 \sin\theta}
{\sin\theta/2} 
\nonumber \\
&=U^{\rm{e}}_{2n-2}(\cos\theta)\cos\theta
 +U^{\rm{o}}_{2n-2}(\cos\theta)(\cos\theta+1).
\label{0beqn:in proof B4 (1)}
\end{align}
Similarly,
\begin{equation}\label{0beqn:in proof B4 (2)}
U^{\rm{o}}_{2n}(\cos\theta)
=U^{\rm{o}}_{2n-2}(\cos\theta)\cos\theta
 +U^{\rm{e}}_{2n-2}(\cos\theta)(\cos\theta-1).
\end{equation}
It then follows from \eqref{0beqn:in proof B4 (1)}
and \eqref{0beqn:in proof B4 (2)} that
\[
U^{\rm{e}}_{2n+2}(x)-2x U^{\rm{e}}_{2n}(x)+U^{\rm{e}}_{2n-2}(x)=0,
\quad
U^{\rm{e}}_{0}(x)=1,
\quad
U^{\rm{e}}_2(x)=2x+1,
\]
and
\[
U^{\rm{o}}_{2n+2}(x)-2x U^{\rm{o}}_{2n}(x)+U^{\rm{o}}_{2n-2}(x)=0,
\quad
U^{\rm{o}}_{0}(x)=1,
\quad
U^{\rm{o}}_2(x)=2x-1.
\]
Comparing with the recurrence relation 
$U_{n+1}(x)-2x U_n(x)+U_{n-1}(x)=0$, we obtain
\eqref{for:ueuoeven}.
On the other hand, using the definition
\[
U_n(\cos\theta)=\frac{\sin(n+1)\theta}{\sin\theta}
\]
and the simple relation
\[
2\cos(n+1)\theta/2
=\frac{2\cos(n+1)\theta/2\,\sin\theta}{\sin\theta}
=\frac{\sin(n+3)\theta/2 - \sin(n-1)\theta/2}{\sin\theta}\,,
\]
we obtain \eqref{for:ueuoodd}.
\end{proof}

\begin{theorem}\label{0bthm:factorization of U}
For $n\ge0$ we have
\begin{align}
U_{2n}(x)&=(U_{n}(x)+U_{n-1}(x))(U_n(x)-U_{n-1}(x)),
\label{for:u-even}\\
U_{2n+1}(x)&=U_n(x)(U_{n+1}(x)-U_{n-1}(x)),
\label{for:u-odd}
\end{align}
\end{theorem}

\begin{proof}
Immediate from
\eqref{0b:U=UeUo} and Theorem \ref{0bthm:Ue and Uo by U}.
\end{proof}

\begin{theorem}\label{0bthm:product expressions}
For $n\ge1$ we have 
\begin{equation}\label{0beqn:factorization of Ue and Uo}
U^{\rm{e}}_n(x)
=\prod_{\substack{1\le l\le n \\ \text{$l$: even}}}
 \bigg(2x-2\cos\frac{l\pi}{n+1}\bigg), 
\quad
U^{\rm{o}}_n(x)
=\prod_{\substack{1\le l\le n \\ \text{$l$: odd}}}
 \bigg(2x-2\cos\frac{l\pi}{n+1}\bigg).
\end{equation}
\end{theorem}

\begin{proof}
We see from \eqref{0beqn:def of Ue} that
the zeros of $U^{\rm{e}}_n(x)$ are given by
$\cos k\pi/(n+1)$ with $1\le k\le n$ being even.
Similarly, by \eqref{0beqn:def of Uo} 
the zeros of $U^{\rm{o}}_n(x)$ are given by
$\cos k\pi/(n+1)$ with $1\le k\le n$ being odd.
Those zeros are all simple since we have
\[
U_n(x)
=\prod_{l=1}^n \bigg(2x-2\cos\frac{l\pi}{n+1}\bigg),
\qquad n\ge1,
\]
and the factorization \eqref{0b:U=UeUo}.
Finally, observing the coefficients of the leading terms
by Theorem \ref{0bthm:Ue and Uo by U},
we obtain \eqref{0beqn:factorization of Ue and Uo}.
\end{proof}

We now go back to \eqref{0beqn:def of Phi_n}.
For the sake of convenience we put 
\begin{align*}
\phi_n(x)
&=\frac14\,\Phi_n(2x) \\
&=((n+1)x^2-3x-n)U_n(x)+(x+1)U_{n-1}(x)+x+1.
\end{align*}
It is noteworthy that $\phi_n(x)$ is factorized
in terms of the Chebyshev polynomials.

\begin{theorem}\label{0bthm:factorization of phi_n(x)}
For $n\ge1$ we have
\begin{equation}
\phi_n(x)=U^{\rm{e}}_n(x) Q_n(x),
\end{equation}
where $U^{\rm{e}}_n(x)$ is the partial Chebyshev polynomial and 
$Q_n(x)$ is given by
\begin{align}
Q_{2n}(x)
&=((2n+1)x^2-3x-2n)(U_{n}(x)-U_{n-1}(x)) 
\nonumber \\
&\qquad +(x+1)(U_{n-1}(x)-U_{n-2}(x)),
\label{0beqn:V_2n}
\\
Q_{2n+1}(x)
&=((2n+2)x^2-3x-2n-1)(U_{n+1}(x)-U_{n-1}(x))
\nonumber \\
&\qquad +(x+1)(U_{n}(x)-U_{n-2}(x)).
\label{0beqn:V_2n+1}
\end{align}
\end{theorem}

\begin{proof}
By Theorem \ref{0bthm:factorization of U} we obtain
\begin{align}
\phi_{2n}(x)
&=((2n+1)x^2-3x-2n)U_{2n}(x)+(x+1)U_{2n-1}(x)+x+1 
\nonumber \\
&=((2n+1)x^2-3x-2n)(U_{n}(x)+U_{n-1}(x))(U_{n}(x)-U_{n-1}(x)) 
\nonumber \\
&\qquad +(x+1)\{U_{n-1}(x)(U_{n}(x)-U_{n-2}(x))+1\}.
\label{0beqn:in proof B.4 (1)}
\end{align}
Using the well known formula $U_{n-1}(x)^2-U_n(x)U_{n-2}(x)=1$,
we have
\begin{align*}
&U_{n-1}(x)(U_{n}(x)-U_{n-2}(x))+1 \\
&\qquad 
=U_{n-1}(x)(U_{n}(x)-U_{n-2}(x))+U_{n-1}(x)^2-U_n(x)U_{n-2}(x) \\
&\qquad
=(U_{n}(x)+U_{n-1}(x))(U_{n-1}(x)-U_{n-2}(x)).
\end{align*}
Thus, two terms in the right-hand side of
\eqref{0beqn:in proof B.4 (1)} have a common factor
$U_{n}(x)+U_{n-1}(x)=U^{\rm{e}}_{2n}(x)$,
we obtain $\phi_{2n}(x)=U^{\rm{e}}_{2n}(x)Q_{2n}(x)$ as desired.
In a similar manner we may show that
$\phi_{2n+1}(x)=U^{\rm{e}}_{2n+1}(x)Q_{2n+1}(x)$.
\end{proof}

We set
\begin{align*}
\Tilde{U}^{\rm{e}}_n(x)
&=U^{\rm{e}}_n\bigg(\frac{x}{2}\bigg)
=\prod_{\substack{1\le l\le n \\ \text{$l$: even}}}
 \bigg(x-2\cos\frac{l\pi}{n+1}\bigg), 
\\
\Tilde{U}^{\rm{o}}_n(x)
&=U^{\rm{o}}_n\bigg(\frac{x}{2}\bigg)
=\prod_{\substack{1\le l\le n \\ \text{$l$: odd}}}
 \bigg(x-2\cos\frac{l\pi}{n+1}\bigg).
\end{align*}
Note that
\[
\Tilde{U}_n(x)=\Tilde{U}^{\rm{e}}_n(x)\Tilde{U}^{\rm{o}}_n(x).
\]
Since $\Tilde{U}_n(x)=U_n(x/2)$ is a monic polynomial
with integer coefficients for all $n\ge0$,
so are $\Tilde{U}^{\rm{e}}_n(x)$ and $\Tilde{U}^{\rm{o}}_n(x)$
by Theorem \ref{0bthm:Ue and Uo by U}.
Moreover, since
\begin{equation}\label{0beqn:factorization of Phi_n(x)}
\Phi_n(x)
=4\phi_n\bigg(\frac{x}{2}\bigg)
=4U^{\rm{e}}_n\bigg(\frac{x}{2}\bigg)Q_n\bigg(\frac{x}{2}\bigg)
=\Tilde{U}^{\rm{e}}_n(x)\cdot 4Q_n\bigg(\frac{x}{2}\bigg)
\end{equation}
by Theorem \ref{0bthm:factorization of phi_n(x)}, we see 
from \eqref{0beqn:V_2n} and \eqref{0beqn:V_2n+1} that
$4Q_n(x/2)$ has integer coefficients too.

We now recall simple facts on the polynomials $\Phi_n(x)$,
which follow by simple observation based on the argument 
in Section \ref{04sec:K1+Pn}.

\begin{proposition}\label{0bprop:01}
Let $n\ge1$.
Then $x=2$ is a double zero of $\Phi_n(x)$,
i.e., $\Phi_n(2)=\Phi_n^\prime(2)=0$ and
$\Phi^{\prime\prime}(2)\neq0$.
Moreover, all zeros of $\Phi_n(x)$ except $x=2$ are real
and in the interval $(-2,2)$.
\end{proposition}

\begin{proposition}\label{0bprop:02}
If $n=1$ or $n\ge3$,
all zeros of $\Phi_n(x)$ except $x=2$ are simple.
For $n=2$ we have $\Phi_2(x)=3(x-2)^2(x+1)^2$.
\end{proposition}

We see Proposition \ref{0bprop:01} that
$\Phi_n(x)$ is divisible by $(x-2)^2$.
On the other hand, $x=2$ is not a zero of $\Tilde{U}^{\rm{e}}_n(x)$ 
by Theorem \ref{0bthm:product expressions}.
Hence it follows from \eqref{0beqn:factorization of Phi_n(x)}
that $4Q_n(x/2)$ is divisible by $(x-2)^2$.
We define a polynomial $R_n(x)$ by
\[
4Q_n\bigg(\frac{x}{2}\bigg)=(x-2)^2R_n(x).
\]
Then $R_n(x)$ has integer coefficients.
Moreover, all zeros of $R_n(x)$ are simple for $n\ge1$
(including $n=2$).
Summarizing the above arguments, we have

\begin{theorem}
For $n\ge1$ we have a factorization:
\[
\Phi_n(x)
=\Tilde{U}^{\rm{e}}_n(x)\cdot 4Q_n\bigg(\frac{x}{2}\bigg)
=(x-2)^2\cdot \Tilde{U}^{e}_n(x)\cdot R_n(x),
\]
where $Q_n(x)$ is given by
\eqref{0beqn:V_2n} and \eqref{0beqn:V_2n+1}.
\end{theorem}

It is interesting to further explore the polynomials $R_n(x)$.
The following table shows the first ten of the polynomials $R_n(x)$.
\begin{center}
\begin{tabular}{|c|r|}
\hline
$n$ & $R_n(x)$ \\ \hline
$1$ & $2x+2$ \\
$2$ & $3x+3$ \\
$3$ & $4x^2+10x+6$ \\
$4$ & $5 x^2+9x+3$ \\
$5$ & $6x^3+18x^2+12x-2$ \\
$6$ & $7x^3+15x^2+2x-7$ \\
$7$ & $8x^4+26x^3+14x^2-20x-14$ \\
$8$ & $9x^4+21x^3-3x^2-26x-7$ \\
$9$ & $10x^5+34x^4+12x^3-54x^2-42x+2$ \\
$10$ & $11x^5 +27x^4 -12x^3 -57x^2 -15x+11$ \\
\hline
\end{tabular}
\end{center}


\end{document}